%% file: main.tex
\documentclass[12pt,a4paper]{article}
\pdfoutput=1
\usepackage{graphics,epsfig,graphicx,color}
\graphicspath{{./Figures/}}
\usepackage{subfigure}

\usepackage{amssymb,latexsym,nth}

\usepackage{setspace}

\usepackage{amsmath}

\usepackage{mathrsfs,amsfonts}

\usepackage{amsthm,amsxtra}

\usepackage[final]{showkeys}

\usepackage{stmaryrd}

\usepackage{algorithm}
\usepackage{algorithmicx}
\usepackage{algpseudocode}
\usepackage[numbers,sort&compress]{natbib}

\newif\ifPDF
\ifx\pdfoutput\undefined
\PDFfalse
\else
\ifnum\pdfoutput > 0
\PDFtrue
\else
\PDFfalse
\fi
\fi

\ifPDF
\usepackage{pdftricks}
\begin{psinputs}
	\usepackage{pstricks}
	\usepackage{pstcol}
	\usepackage{pst-plot}
	\usepackage{pst-tree}
	\usepackage{pst-eps}
	\usepackage{multido}
	\usepackage{pst-node}
	\usepackage{pst-eps}
\end{psinputs}
\else
\usepackage{pstricks}
\fi

\ifPDF
\usepackage[debug,pdftex,colorlinks=true, 
linkcolor=blue, bookmarksopen=false,
plainpages=false,pdfpagelabels]{hyperref}
\else
\usepackage[dvips]{hyperref}
\fi

\usepackage[openbib]{currvita}

\usepackage{fancyhdr}

\newtheorem{theorem}{Theorem}[section]
\newtheorem{lemma}[theorem]{Lemma}

\newcommand{\eps}{\varepsilon}

\newcommand{\pdr}[2]
{\dfrac{\partial{#1}}{\partial{#2}}}

 \newcommand{\bbS}{\mathbb S}

\newcommand{\bzero}{{\mathbf 0}}
\newcommand{\btheta}{{\boldsymbol\theta}}

\newcommand{\bnu}{{\boldsymbol \nu}}

\newcommand{\bxi}{\boldsymbol \xi}
\newcommand{\bzeta}{\boldsymbol \zeta}
\newcommand{\bEta}{\boldsymbol \eta}

\newcommand{\bk}{\mathbf k} 
 \newcommand{\bn}{\mathbf n}
 \newcommand{\bp}{\mathbf p}
\newcommand{\bq}{\mathbf q} 
 
\newcommand{\bu}{\mathbf u} \newcommand{\bv}{\mathbf v} 
\newcommand{\bw}{\mathbf w} \newcommand{\bx}{\mathbf x} 
\newcommand{\by}{\mathbf y} \newcommand{\bz}{\mathbf z}

\DeclareMathOperator*{\argmin}{\arg\min} 
\usepackage[usenames,dvipsnames]{pstricks}
\usepackage{epsfig}
\usepackage{pst-grad} 
\usepackage{pst-plot} 
\usepackage[space]{grffile} 
\usepackage{etoolbox} 

\makeatletter
\newcommand{\chapterauthor}[1]{%
	{\parindent0pt\vspace*{-25pt}%
		\linespread{1.1}\large\scshape#1%
		\par\nobreak\vspace*{35pt}}
	\@afterheading%
}
\makeatother

\author{Hongkai Zhao\quad Yimin Zhong}

\title{A hybrid adaptive phase space method for reflection traveltime tomography}
\begin{document}
\maketitle
\begin{abstract}
We present a hybrid imaging method for a challenging travel time tomography problem which includes both unknown medium and unknown scatterers in a bounded domain. The goal is to recover both the medium and the boundary of the scatterers from the scattering relation data on the domain boundary. Our method is composed of three steps: 1) preprocess the data to classify them into three different categories of  measurements corresponding to non-broken rays, broken-once rays, and others, respectively, 2) use the the non-broken ray data and an effective  data-driven layer stripping strategy--an optimization based iterative imaging method--to recover the medium velocity outside the convex hull of the scatterers, and 3) use selected broken-once ray data to recover the boundary of the scatterers--a direct imaging method. By numerical tests, we show that our hybrid method can recover both the unknown medium and the not-too-concave scatterers efficiently and robustly. 
\end{abstract}
\section{Introduction}
Traveltime tomography is an important class of inverse problems which appear in various applications such as global seismology~\cite{kennett1991traveltimes,inoue1990whole,bishop1985tomographic,clarke20013d,minkoff1996computationally}, ocean acoustic tomography~\cite{munk1979ocean, munk2009ocean,collins1994inverse,jensen2011computational}, ultrasound tomography~\cite{hormati2010robust,schomberg1978improved,jin2006thermoacoustic} in biomedical imaging and so on. It determines the internal velocity of the medium by measuring the wave traveltime between points on the boundary.

Theoretically, the traveltime tomography is very closely related to boundary rigidity and lens rigidity problems in differential geometry~\cite{croke1991rigidity,frigyik2008x,stefanov2004stability,stefanov2005boundary,stefanov2008boundary,guillarmou2017lens,kurylev2010rigidity,pestov2004characterization,stefanov2012geodesic,pestov2005two}.  The boundary rigidity problem is to determine the metric of compact Riemannian manifold up to a diffeomorphism from first arrival time information, and the traveltime is the length of geodesic, which is also called ray in geometric optics context, connecting two points on boundary. The lens rigidity problem utilizes multiple arrival times information to determine the Riemann metric. The multiple arrival times are encoded in scattering relation which consists of incoming and outgoing points and directions as well as the traveltime.

For boundary rigidity, the uniqueness of reconstruction (up to an action of a diffeomorphism) is known for simple metrics (see~\cite{croke1991rigidity,mukhometov1981problem,michel1981rigidite,stefanov2005boundary,stefanov2005recent} and references therein) and many other cases~\cite{gromov1983filling,besson1995entropies,croke1990rigidity,lassas2003semiglobal,sharafutdinov1994integral}. A compact Riemannian manifold $(M,\partial M, g)$ is \emph{simple} if the boundary $\partial M$ is strictly convex with respect to its metric $g$ and there are no conjugate points along any geodesic. Moreover, for simple manifolds, the knowledge of scattering relation does not provide more information than boundary distance function. See~\cite{stefanov2007local,stefanov2008boundary,stefanov2008microlocal,guillarmou2017lens,stefanov2008integral} and references therein for recent progress on lens rigidity for non-simple manifolds. Numerically, there are many numerical algorithms motivated by the theoretical progress in boundary rigidity and lens rigidity problems,
see~\cite{chung2011adaptive,chung2008phase,chung2007new,leung2006adjoint,leung2007transmission,li2013fast,li2014level,glowinski2015penalization} for algorithmic developments. 

When there are strong scattering effects or impenetrable obstacles inside the medium, then the geodesics could be broken. In~\cite{kurylev2010rigidity}, Kurylev, Lassas and Ulhmann established a uniqueness result for reconstructing Riemannian metric from the broken scattering relation.  For reflective obstacles,  we consider an incident ray jointed with its corresponding reflected ray as a broken geodesic by imposing reflection condition at the joint point. When there is only one strictly convex obstacle inside the manifold, then under certain conditions such as simple manifolds of dimension $\ge 2$ with real analytic metric~\cite{krishnan2009support} or manifolds of dimension $\ge 3$ with convex hypersurface foliation~\cite{uhlmann2016inverse},  the Riemannian metric outside the obstacle can be uniquely recovered from all nonbroken rays by Helgason support theorem~\cite{helgason2011radon}. However, all of the proofs of uniqueness are not constructive, and another difficulty in practice would be how to efficiently distinguish the broken and nonbroken scattering relation in measurements. In~\cite{chung2011adaptive}, Chung, Qian, Ulhmann and Zhao proposed a numerical reconstruction algorithm which is able to distinguish nonbroken and broken rays by measuring mismatch in scattering relation data during each iteration, if a broken ray is falsely predicted as nonbroken one, then there could be an $O(1)$ mismatch in the data. 

In~\cite{chung2007new,chung2008phase,chung2011adaptive}, the authors have developed a phase-space approach for transmission and reflection traveltime tomography for acoustic and elastic media by using the Stefanov-Ulhmann identity formulated in~\cite{stefanov1998rigidity}. The method is advantageous over  traditional methods in inverse kinematic problems~\cite{romanov1987inverse,bishop1985tomographic,sei1994gradient,sei1995convergent,washbourne2002crosswell}, because it uses multiple arrival times systematically and has the potential to handle anisotropic metrics as well, while these traditional methods can only recover isotropic metrics by utilizing first arrival times. 

However, for the challenging case where both the medium and the scatterers are unknown, adaptive phase space method developed in \cite{chung2011adaptive} only uses non-broken ray to recover the medium outside the convex hull of the scatterers. In this work, we combine the adaptive phase space method, which is an optimization based iterative method, with a direct imaging method using selected broken-once ray data. This will give us the possibility to recover non-convex part of the boundary of the unknown scatterers. We also make several improvements that include preprocessing of the scattering relation data to classify data corresponding to non-broken rays, broken-once rays, and the others respectively and  improvements in efficiency and robustness for the adaptive phase method. 

Although Stefanov-Ulhmann identity can be used as feedback from computed metric to the exact metric, however, the identity itself is nonlinear. The linearization of Stefanov-Ulhmann identity will require the two metrics to be close enough, therefore the initial guess is critical for stable reconstruction. In our method, we first consider those geodesics with short traveltimes, by taking Taylor expansion for such geodesics, we can obtain  Dirichlet and Neumann data of the metric on boundary, then we can extrapolate the initial guess of metric from these boundary data and the initial guess should be quite close to the exact solution near boundary. For the construction, our method also follows the layer stripping idea, but quite different from \cite{chung2011adaptive}, which selects the rays according to smallness in mismatch and could end up with some long rays which may deviate from layer stripping process. In our method, we introduce an auxiliary fidelity function to guide the layer stripping process. It can seen from our numerical experiments in Section~\ref{sec:num} that the iteration number can be reduced and the reconstruction process is very stable. For the reflection traveltime tomography,  the method in~\cite{chung2011adaptive} will take more iterations and more time due to its trial and error strategy in distinguishing broken and nonbroken rays, while our method first preprocess the data and directly detect non-broken rays from the scattering relation by scanning discontinuities in derivatives. The non-broken rays can immediately be used to reconstruct the metric outside the convex hull of obstacles by Helgason support theorem~\cite{helgason2011radon}. Furthermore, when the obstacles are not large and not too concave or the metric does not vary too much near obstacle, then our method can be used to capture non-convex shape of the obstacles by tracking those rays which hit the obstacle in normal direction. Such rays will reverse their trace back to their initial location after reflection and provide a direct and stable way of locating points on boundaries of obstacles by tracing the ray to half of the traveltime (see the numerical experiments in Section~\ref{sec:num}). 


The paper is organized as follows: we introduce the mathematical formulation for reflection traveltime tomography and broken geodesics in Section~\ref{sec:2}. Then we describe our numerical algorithm and the hybrid method in Section~\ref{sec:3}. Test results of our method for different setups are presented in Section~\ref{sec:num}.

\section{Mathematical formulation for reflection traveltime tomography}\label{sec:2}
\subsection{Broken scattering relation}
Let $(M,g)$ be a compact Riemann manifold with boundary dimension of $d$, and denote $S(M)$ its unit tangent bundle. The scattering relation or lens relation~\cite{kurylev2010rigidity} is 
\begin{equation}
\begin{aligned}
\mathcal{L} = \{((\bx,\bxi), (\by, \bzeta), t)\in S(M)\times S(M)\times \mathbb{R}_{+}\cup \{0\}: \bx,\by \in \partial M, \\(\gamma_{\bx,\bxi}(t), \dot{\gamma}_{\bx,\bxi}(t)) = (\by,\bzeta)\text{ for some }t\ge 0\},
\end{aligned}
\end{equation}
where $\gamma_{\bx,\bxi}$ is the geodesic of $(M,g)$ starts from $\bx$ with direction $\bxi$ at $t=0$.

As defined in~\cite{kurylev2010rigidity}, a broken-once geodesic is a path $\alpha = \alpha_{\bx, \bxi, \bz,\bEta}(t)$ where $\bz = \gamma_{\bx, \bxi}(s) \in M$ for some $s\ge 0$, $\bEta\in S_{\bz}(M)$, and 
\begin{equation}
\alpha_{\bx,\bxi, \bz, \bEta}(t) = \begin{cases}
\gamma_{\bx,\bxi}(t),& \text{ for } t < s,\\
\gamma_{\bz,\bEta}(t- s), & \text{ for } t\ge s.            
\end{cases}
\end{equation}
The entering and exiting points of broken geodesics define the broken scattering relation~\cite{kurylev2010rigidity}
\begin{equation}
\begin{aligned}
\mathcal{R} = \{((\bx,\bxi), (\by, \bzeta), t)\in S(M)\times S(M)\times \mathbb{R}_{+}\cup\{0\}: (\bx,\bxi)\in \Gamma_{+}, (\by,\bzeta)\in \Gamma_{-}, \\t = l(\alpha_{\bx,\bxi, \bz,\bEta}),\text{ and } (\alpha_{\bx,\bxi, \bz,\bEta}(t), \dot{\alpha}_{\bx,\bxi, \bz,\bEta}(t)) = (\by,\bzeta) \text{ for some } (\bz,\bEta)\in S(M)\},
\end{aligned}
\end{equation}
where $ l(\alpha_{\bx,\bxi, \by,\bEta}) \in \mathbb{R}_{+}\cup\{\infty\}$ denotes the smallest $l>0$ that $\alpha_{\bx,\bxi, \by,\bEta} \in\partial M$. Let $\bnu$ be the interior unit normal vector  of $\partial M$ and we define the following incoming and outgoing subbundles:
\begin{equation}
\begin{aligned}
&\Gamma_{+} = \{(\bx, \bxi)\in S(M): \bx\in \partial M, \langle \bxi, \bnu\rangle_g > 0\},\\
&\Gamma_{-} = \{(\bx, \bxi)\in S(M): \bx\in \partial M, \langle \bxi, \bnu\rangle_g < 0\}.
\end{aligned}
\end{equation} 
Note that the scattering relation does not contain any information about the point $\bz$ or its corresponding direction $\bEta$ where the broken ray $\alpha_{\bx,\bxi, \bz,\bEta}$ changes its direction~\cite{kurylev2010rigidity}.

\subsection{Mathematical formulation}
Let $\Omega\subseteq \mathbb{R}^d$ be a compact domain and let $(g_{ij})$ be a Riemann metric on it. We define the Hamiltonian $H_g$ by
\begin{equation}
H_g(\bx,\bxi) = \frac{1}{2}\left(\sum_{1\le i, j\le d}g^{ij}(\bx)\xi_i\xi_j - 1\right),
\end{equation}
where $(g^{ij}) =(g_{ij})^{-1}$. Let $X^{(0)} = (\bx^{(0)}, \bxi^{(0)})$ be the initial condition belonging to the inflow set: 
\begin{equation}
\label{eq:inflow}
\mathcal{S}^{-} = \{(\bx,\bxi) \,|\, \bx\in\partial\Omega, H_g(\bx,\bxi) = 0,\, \sum_{1\le i,j\le d}g^{ij}\xi_i\nu_j < 0\},
\end{equation}
where $\bnu(x)$ is the unit outward normal vector at $\bx\in\partial\Omega$ and $\nu_j(\bx)$ is the $j$-th component of $\bnu(\bx)$. The geodesic $X_g(s, X^{(0)})$ satisfies following Hamiltonian system:
\begin{equation}\label{eq:hamiltonian}
\frac{d\bx}{ds} = \partial_{\bxi}H_{g},\quad \frac{d\bxi}{ds} = -\partial_{\bx}H_{g},
\end{equation}
with initial condition $(\bx(0),\bxi(0)) = X^{(0)}$. Then the solution $X_g(s, X^{(0)}) = (\bx(s),\bxi(s))$ defines a geodesic (or a ray) in phase space, where $\bx(s)$ is the projection onto physical space $\Omega$, $\bxi(s)$ is the cotangent vector at $\bx(s)$, and $s$ denotes the traveltime.

In the following sections, we consider the case in which there are obstacles inside the domain $\Omega$. Rays are broken and reflected at the boundary when they hit the obstacles. Then the Hamiltonian system~\eqref{eq:hamiltonian} needs to impose a jump condition at the reflection point. The jump condition for the case when there is only one obstacle strictly lying in $\Omega$ has been derived in~\cite{chung2011adaptive}.  Let $\Gamma$ be the interface where rays are reflected, for each reflected ray, there is a unique time $s^{\ast} > 0$ that $\bx(s^{\ast})\in\Gamma$ hits the interface at an incoming direction $\bxi_{\texttt{in}}:=\bxi(s^{\ast})$. The ray will be reflected to an outgoing direction $\bxi_{\texttt{out}}:=(I - 2\bn\bn^T)\bxi_{\texttt{in}}$, where $\bn$ is unit outward normal vector at $\bx(s^{\ast})$ of the obstacle. The Hamiltonian system for a broken-once ray will be
\begin{equation}
\begin{aligned}
&\frac{d\bx}{ds} = \partial_{\bxi}H_{g},\quad \frac{d\bxi}{ds} = -\partial_{\bx}H_{g},\quad 0<s\le s^{\ast}, \quad (\bx(0), \bxi(0)) = X^{(0)}, \\
&\frac{d\bx}{ds} = \partial_{\bxi}H_{g},\quad \frac{d\bxi}{ds} = -\partial_{\bx}H_{g},\quad s>s^{\ast},\quad (\bx(s^{\ast}),\bxi(s^{\ast})) = (\bx(s^{\ast}),\bxi_{\texttt{out}}).
\end{aligned}
\end{equation}
To derive Stefanov-Ulhmann identity, we need the following Jacobian matrix with respect to the initial condition
\begin{equation} 
J_g(s, X^{(0)}) := \frac{\partial X_g}{\partial X^{(0)}}(s, X^{(0)}) = \begin{pmatrix}
\frac{\partial \bx}{\partial \bx(0)} & \frac{\partial\bx}{\partial \bxi(0)}\\
\frac{ \partial \bxi}{\partial \bx(0)} & \frac{\partial\bxi}{\partial\bxi(0)}
\end{pmatrix}.
\end{equation}
Let 
\begin{equation}
M = \begin{pmatrix}
H_{\bxi,\bx} & H_{\bxi,\bxi}\\
-H_{\bx,\bx} & -H_{\bx,\bxi}
\end{pmatrix},
\end{equation}
then $J_g(s, X^{(0)})$ satisfies system
\begin{equation}\label{eq:jacobian}
\begin{aligned}
&\frac{dJ}{ds} = MJ,\quad  J(0) = I \quad \text{ for } 0 < s < s^{\ast},\\
&\frac{dJ}{ds} = MJ,\quad J(s^{\ast}) = B\quad \text{ for } s>s^{\ast}.
\end{aligned}
\end{equation}
where 
\begin{equation}\label{eq:jump-condition}
\begin{aligned}
B =  \begin{pmatrix}
J(s^{\ast})_{11} & J(s^{\ast})_{12}\\
\partial_{\bxi}\bxi_{\texttt{out}} J(s^{\ast})_{21} + \partial_{\bx} \bxi_{\texttt{out}}J(s^{\ast})_{11} &	\partial_{\bxi}\bxi_{\texttt{out}} J(s^{\ast})_{22} + \partial_{\bx}\bxi_{\texttt{out}}J(s^{\ast})_{12} 
\end{pmatrix}.
\end{aligned}
\end{equation}
Similar to~\cite{chung2011adaptive}, we consider function 
\begin{equation}
F(s) = X_{g_2}(t-s, X_{g_1}(s, X^{(0)})),
\end{equation}
where $t = t_{g_1}$ and 
\begin{equation}
\int_0^{t} F'(s) ds = X_{g_1}(t,X^{(0)}) - X_{g_2}(t, X^{(0)}).
\end{equation}
The left-hand side is 
\begin{equation}
\int_0^{t} F'(s) ds = \int_0^t J_{g_2}(t-s, X_{g_1}(s,X^{(0)})) \times (V_{g_1} - V_{g_2})(X_{g_1}(s, X^{(0)})) ds,
\end{equation}
where $V_g = \left(\partial_{\bxi}H_g, -\partial_{\bx}H_g \right)$. Linearize above integral's right-hand side at metric $g_2$, then we approximately have
\begin{equation}\label{eq:linearized-su}
\begin{aligned}
X_{g_1}(t,X^{(0)}) - &X_{g_2}(t, X^{(0)}) \simeq \\&\int_0^t J_{g_2}(t-s, X_{g_2}(s,X^{(0)})) \times \partial_g V(g_1-g_2, g_2,X_{g_2}(s, X^{(0)})) ds.
\end{aligned}
\end{equation}
For simplicity, we only consider isotropic metric in following sections. Let $X=(\bx,\bxi)$, then
\begin{equation}
g_{ij} = c^{-2}\delta_{ij},\quad \partial_g V(\lambda, g, X) = (2c\lambda \bxi, -(\lambda \nabla c + c\nabla \lambda )|\bxi|^2).
\end{equation}

\section{A hybrid method for reconstruction}\label{sec:3}
\subsection{Stabilized adaptive phase space method}\label{sec:stab-adaptive}
In~\cite{chung2011adaptive}, the authors have introduced the adaptive phase space method. The numerical method is an iterative algorithm based on linearized Stefanov-Uhlmann identity~\eqref{eq:linearized-su}, and the algorithm automatically follows layer-stripping process by choosing those rays with small mismatches on exiting phase measurements. However, using \emph{only} mismatch information could deviate from layer-stripping since small mismatches do not guarantee small errors on paths. Hence we propose a stabilized iterative method to overcome the ``false picking".

For simplicity, we first introduce the method for the medium without interior obstacle. The metric $g$ is discretized over an underlying Eulerian grid in the physical domain $\Omega$. The linearized Stefanov-Ulhmann identity~\eqref{eq:linearized-su} is discretized along the ray for each initial condition $X^{(0)}$ in phase space, the Jacobian matrix along the ray is computed by~\eqref{eq:jacobian}, the value of metric $g$ on non-grid points are linearly interpolated from neighborhood grid values. Therefore each integral equation along the ray $X_g(s,X^{(0)})$ represents a linear equation for neighboring  grid values.

Let $X_k^{(0)}, k = 1,2,\dots, m$, be initial coordinates in phase space of those measurements $X_g(t_k, X_k^{(0)})$, where $t_k$ is the traveltime of corresponding ray starts from $X_k^{(0)}$. We then iteratively construct a sequence of metric $g_n$  as follows.

First, we need to construct a good initial guess $g_0$ for the linearized problem, which is important for convergence of the metric.

For short geodesic $X_g(s, X^{(0)}) = (\bx(s), \bxi(s)), 0\le s \ll 1$,  with assumption on differentiability and local analyticity of $g$, we can easily deduce that
\begin{equation}
\begin{aligned}
\bx(s) &= \bx^{(0)} + s H_{g,\bxi} + \frac{s^2}{2} (H^{(0)}_{g,\bxi,\bx}H^{(0)}_{g,\bxi} - H^{(0)}_{g,\bxi,\bxi}H^{(0)}_{g,\bx}) + O(s^3),\\
\bxi(s) &= \bxi^{(0)} - sH^{(0)}_{g,\bx} - \frac{s^2}{2}(H^{(0)}_{g,\bx,\bx} H^{(0)}_{g,\bxi} - H^{(0)}_{g,\bx,\bxi}H^{(0)}_{g,\bx}) + O(s^3),
\end{aligned}
\end{equation}
where $H^{(0)}_g = H_g(X^{(0)})$. In case of isotropic metric, $g^{ij} = c^2(\bx)\delta_{ij}$, choose $X^{(0)} = (\bx^{(0)}, c^{-1}(\bx^{(0)})\bv)$ for a unit direction $\bv$ that is closest to tangential direction, and we have the following Talyor expansions
\begin{equation}
\begin{aligned}
\bx(s) &= \bx^{(0)} + s c \bv + \frac {s^2}{2}c (2\bv(\nabla c)^T \bv - \nabla c) + O(s^3),\\
\bxi(s) &= \bxi^{(0)} - s \frac{\nabla c}{c} - \frac{s^2}{2}((\nabla c\nabla c^T + c\nabla^2 c) c\bv - \nabla c \bv^T \frac{\nabla c}{c}) + O(s^3).
\end{aligned}
\end{equation}
By ignoring the $O(s^3)$ term of the first approximation above at a boundary point $\bx^{0}\in\partial\Omega$, we can solve $\nabla c(\bx_0)$ from the following linear equation
\begin{equation}
\bx(s) = \bx^{(0)} + s c \bv + \frac {s^2}{2}c (2\bv(\nabla c)^T \bv - \nabla c).
\end{equation}
After solving $\nabla c$ for all boundary points, we can construct a smooth initial guess $g_0^{ij} = c_0^{2}(\bx)\delta_{ij}$ by solving following minimization problem
\begin{equation}
c_0 = \argmin_{\tilde{c}} \frac{1}{2}\int_{\partial\Omega} \left(|\nabla c - \nabla \tilde{c}|^2 + |c - \tilde{c}|^2 \right)ds + \frac{\gamma}{2}\int_{\Omega} |\nabla \tilde{c}|^2 dx,
\end{equation}
where $\gamma$ is a small regularization parameter.
Define measurement mismatch $d^n_k$ of $k$-th ray as
\begin{equation}
d_k^n = X_g(t_k, X_k^{(0)}) - X_{g_n}(t_k, X_k^{(0)}).
\end{equation}
By linearized Stefanov-Ulhmann identity, we define linear operator $F_k^n$ along the $k$-th ray
\begin{equation}
\label{eq:SU}
F_k^n \tilde{g} = J_{g_n}(t,  X_k^{(0)})\int_0^t J_{g_n}^{-1}(s, X_k^{(0)})\partial_g V_{g_n}(\tilde{g}, X_{g_n}(s, X_k^{(0)})) ds.
\end{equation}
The linear operators $F_k^n$ are matrices of size $4\times N$, where $N$ is the number of unknowns in physical domain $D$, $d^n_k$ are vectors of size $4\times 1$. For each $g^n$, we define block matrix $A^n$ of size $m \times N$, with each block of size $4\times 1$ that
\begin{equation}\label{eq:linear-eq}
A^n =  \begin{bmatrix}
F^n_1 \\  F^n_2 \\ \vdots \\ F_m^n
\end{bmatrix}, \mbox{ and } A^n g \simeq b^n:=\begin{bmatrix}
d_1^n \\ d_2^n \\\vdots \\ d_m^n 
\end{bmatrix}.
\end{equation}

From a geometrical viewpoint, if the matrix $F_k^n$'s $l$-th column is nonzero, then it means the $k$-th ray passes nearby the $l$-th grid, the total number of nonzero columns in $F_k^n$ roughly represents the length of this ray. During each iteration, suppose the $k$-th ray's mismatch is negligible, then it is highly possible that the matrix $F_k^n$ is close to the correct one. In other words, the velocity field at those grid points used in the computation of the linear operator $F_k^n$ in \eqref{eq:SU} along the $k$-th ray is likely to be accurate. In order to characterize this property,  we define a fidelity function $0 \le p^n(\bx)\le 1$ for all grid points $ \bx\in \Omega$ at $n$-th iteration, which approximately represents the confidence of the current value of $g$ at $\bx$ and initially $p^0(\bx)\equiv 0$ over $\Omega$.
We also define a residual function for $F_k^n$ of matrix $A^n$, 
\begin{equation}
\texttt{res}(F_k^n) = \texttt{nnz}(F_k^n) - \sum_{i:F_k^n(:,i)\neq 0} p(\bx_i), \quad 
\end{equation}
where $\texttt{nnz}(F_k^n)$ is the number of nonzero columns in $F_k^n$ and $\bx_i, i=1, 2, \ldots, N$ are the grid points.
Geometrically, $\texttt{res}(F_k^n)$ approximately reveals effective length of the part of the $k$-th ray along which the velocity field is unknown. Or simply an indicator of the accuracy of the linearized Stefanov-Ulhmann identity \eqref{eq:linearized-su} along the $k$-th ray. At $n$-th iteration, we use $S_n$ to represent a subset of row blocks, which corresponds to all indices $k$ that $\texttt{res}(F_k^n)$ are under some effective length threshold $r_{\min}$,  
\begin{equation}
S_n = \{k\in \{1, 2, \dots, m\} |\, \texttt{res}(F_k^n) \le r_{\min}\}.
\end{equation}
We notice that the rays belonging to $S_n$ will provide more stable reconstruction than other rays do.

At $n$-th iteration, we construct a perturbation $\tilde{g}$ by minimizing functional
\begin{equation}\label{eq:regularize}
\mathcal{H}(\tilde{g}) = \frac{1}{2} \|A^n(S_n, :) \tilde{g} - b(S_n)\|^2 + \frac{\beta}{2}\|\nabla \tilde{g}\|^2.
\end{equation}
where $\beta$ is a regularization parameter. We then update $g^{n+1}$ by
\begin{equation}
g^{n+1} = g^n + \tilde{g},
\end{equation}
until the relative error of measurement is below certain tolerance level. The fidelity function $p^{n+1}$ is updated by the following steps. 
\begin{enumerate}
	\item For each $k\in S_n$, we calculate the residual error along the $k$-th ray $$e_k = \|A^n(k, :)\tilde{g} - b(k)\|.$$
	\item If $e_k$ is under some threshold $\tau$, we update the fidelity for grid points involved in the linear operator $F^n_k$ by
	\begin{equation}
	p^{n+1}(\bx_i) = \max(p^n(\bx_i), 1 - \alpha e_k),\quad\text{ when } A^n(k,i)\neq \mathbf{0}_{4\times 1}.
	\end{equation} 
	where the threshold $\tau$ is chosen to be small and $\alpha$ is a parameter to control the fidelity decay.  Here we have made assumption that if an effectively ``short'' ray has smaller mismatch with measurement, then metric $g$'s value along the ray has higher fidelity.
\end{enumerate} 

This adaptive method follows a layer-stripping process by using shorter rays near-boundary first and stripping them by updating fidelity function, which corresponds to the foliation process introduced in~\cite{uhlmann2016inverse}. The layers are implicitly charaterized by $p^n$ at each iteration, and only depend on the measurements/data. We can see from the numerical experiments in Section~\ref{sec:num} that this method not only improves stability and efficiency over the previous adaptive phase space method introduced in \cite{chung2011adaptive}, but also achieves a better accuracy on reconstructed metric.

\subsection{Interior reflection detection}\label{sec:ref-detect}
When there is a reflective obstacle lying inside the medium, and neither the obstacle nor the metric is known, then the reconstruction of both can be very challenging. In~\cite{chung2011adaptive}, the authors used the adaptive phase space method to distinguish most of the unbroken rays. These rays can help to recover the convex hull (under the same metric) of the unknown obstacle and the medium outside the convex hull. However, for concave part of the boundary of the obstacle, one also needs to use those broken rays. In the following we propose a direct imaging method to find points on the boundary of the obstacle using broken-once rays and the metric $g$ recovered from the adaptive phase method based on non-broken rays as described in the previous section.


Let's consider a simple case in this scenario. Suppose the physical domain $\Omega$ is \emph{convex} and interior reflector $D\subseteq \Omega$ have $C^2$ boundaries, which are homotopic equivalent to $S^{d}$. 
We call a point $\by\in\partial D$ a \emph{radiative} boundary point if a ray starts from $\by$ with direction along the outward normal $\bn(\by)$ of $\partial D$ only intersects with $\partial\Omega$. Denote the set of all radiative boundary points by $R$, assume $R$ is nonempty, then consider a continuous mapping $L_g:R\to \partial\Omega$ that
\begin{equation}
L_g(\by)= P_{\bx}X_g(t, (\by, c^{-1}\bn(\by))),
\end{equation}
where $t$ is the traveltime such that $X_g(t, (\by, c^{-1}\bn(\by))) \in \Gamma_{-}$ and $P_{\bx}$ is projection mapping from phase space to physical space. Let the range of $L_g$ be $T_g$, then for each $\bx^{(0)}\in T_g$, there exists a direction $\bxi^{(0)}$ and travel time $s^{\ast}(\bx^{(0)},\bxi^{(0)}) $ such that  $ \bx(s^{\ast})\in \partial D, \bxi_{\texttt{out}} = -\bxi(s^{\ast})$, which means the ray hits the obstacle in the inward normal direction. Such reflected rays will go back to its initial physical location and have exact opposite directions. Though numerically such rays are rare, but they are very stable and can be used to parametrize the obstacle implicitly, because the reflection occurs exactly at halfway of a broken-once ray, see Figure~\ref{fig:orthorays}.
\begin{figure}[!htb]
	\centering
	\includegraphics[scale=0.45]{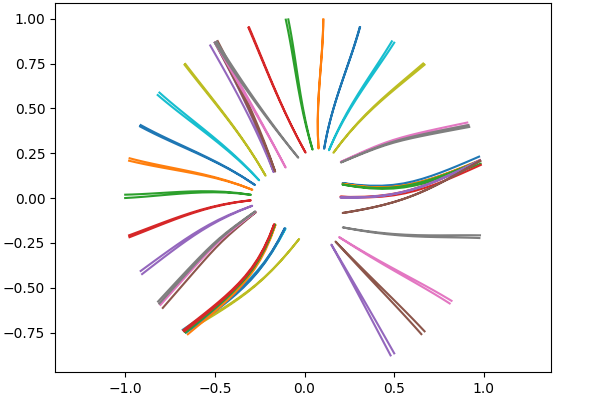}
	\caption{The broken rays hit the obstacle in (nearly) normal direction.}
	\label{fig:orthorays}
\end{figure}

We remark that once such rays exist, then an obstacle will be detected because otherwise these non-broken rays (or geodesics) are not unique along the their initial phases.

Numerically, let $X^{(0)}_k, k=1,2,\dots,m$ be initial coordinates in phase space as previous section, $t_k$ is the traveltime of corresponding ray $(\bx_k(s), \bxi_k(s))$ starts from $X^{(0)}_k=(\bx_k^{(0)}, \bxi_k^{(0)})$. We use $T$ to represent a subset of index $\{1,2,\dots, m\}$ of rays satisfying
\begin{equation}
\label{eq:broken}
T = \Big\{k\,\Big|\,\|\bx_k^{(0)} - \bx_k(t_k)\|_{L^2} + \|\bxi_k^{(0)} + \bxi_k(t_k)\|_{L^2} < \epsilon\Big\}.
\end{equation}
from the given data, i.e., the scattering relations for certain numerical tolerance $\epsilon>0$. These rays are considered to be hitting the reflector in the normal direction. And the reflection happens at middle point $s^{\ast}_k = t_k/2$. Comparing to other broken rays, these are more predictable and easier to use to locate the boundary due to the knowledge of reflection time and angle.

\subsection{Non-broken rays detection}\label{sec:non-broken-detect}
The adaptive phase space method~\cite{chung2011adaptive} has shown its potential to distinguish most of the non-broken rays during the layer-stripping process, if a broken ray or a non-broken ray is falsely predicted, then it is likely to produce an $O(1)$ mismatch in scattering relation due to the jump condition~\eqref{eq:jump-condition}. But this method would have difficulties in critical cases, e.g. distinguishing near-tangent broken and non-broken rays. 

For each physical location $\bx^{(0)}$ on the boundary of the domain, $\partial\Omega$, we define the set of directions for broken-rays (including tangential rays) by
\begin{equation}
\mathcal{B}(\bx^{(0)}) = \{ \bxi~|~ (\bx^{(0)},\bxi)\in \mathcal{S}^-, X_g(t, \bx^{(0)}, \bxi) \cap  \partial D \neq \emptyset \text{ for some } t > 0\}.
\end{equation}
 Let $\mathcal{C}(\bx^{(0)})$ be the smallest simply connected set containing $\mathcal{B}(\bx^{(0)})$, then we only have to find out the directions $\bxi\in \partial \mathcal{C}(\bx^{(0)})$, since the rays with directions outside $\mathcal{C}(\bx^{(0)})$ are all non-broken.  In order to be able to detect the set $\partial \mathcal{C}^{(0)}$ with both medium and object unknown, we need to make a few further assumptions on the metric and obstacle. For simplicity, we focus our study on the case where the medium is isotropic, i.e., $H(\bx, \bxi)=\frac{1}{2}(c^2(\bx)\|\bxi\|^2-1)$, although our results can be extended to the more general anisotropic case.
\begin{enumerate}
	\item $c(\bx)$ is a $C^3$ function and non-trapping.\label{as:1}
	\item Both boundaries $\partial \Omega$ and $\partial D$ are $C^2$. Without loss of generality, we assume the boundaries $\partial\Omega$ and $\partial D$ are represented by $G(\bx) = 0$ and $F(\bx) = 0$ respectively, where $G$ and $F$ are $C^2$ functions. \label{as:2}
	\item If a broken ray $X_g(s, \bx^{(0)}, \bxi^{(0)}) = (\bx(s), \bxi(s))$ is tangential to $\partial D$ and its exiting phase is $(\by, \bzeta) = (\bx(t), \bxi(t))$, where $t$ is traveltime, then there is a constant $\delta >0 $ that $|\bzeta \cdot \bn(\by)| > \delta$, where $\bn(\by)$ is the outward unit normal vector of $\partial\Omega$ at $\by$. Physically speaking, this means for those the rays that are tangential to the obstacle's boundary, they exit in non-tangential directions. \label{as:3}
	\item If a broken ray $X_g(s, \bx^{(0)}, \bxi^{(0)}) = (\bx(s), \bxi(s))$  is tangential to $\partial D$ at $\bp = \bx(t_{\bp})\in\partial D$, 
	\begin{equation}\label{eq:as4}
\bn_{\bp}\cdot \frac{\partial \bxi}{\partial s}\Big|_{t_{\bp}} +\left( \frac{\partial \bx}{\partial s}\nabla\frac{\nabla F}{\|\nabla F\| }\bxi\right)\Big|_{t_{\bp}}\neq 0,
	\end{equation}
where $\bn_{\bp}$ is outward unit normal at $\bp$. This term stands the interaction between the geometry of the obstacle and the medium. When this term is nonzero, reflection at the obstacle boundary would impose a significant change on an impinging ray direction compared to the change of direction due to the medium variation. In the isotropic case, this requirement is equivalent to
	\begin{equation}
	-\bn_{\bp}\cdot \frac{\nabla c(\bp)}{c(\bp)} + \bk\cdot \frac{\textrm{Hess}(F)}{\|\nabla F\|}\Big|_{\bp}\bk \neq 0,
	\end{equation}
	where $\bk = \frac{\bxi}{\|\bxi\|}$. The second term is bounded below by the minimal principal curvature $\lambda_{\min}$, therefore, if we assume that 
	\begin{equation}
	\frac{\|\nabla c(\bp)\|}{c(\bp)} < \lambda_{\min},
	\end{equation}	
	then~\eqref{eq:as4} is satisfied. Simply speaking, if the obstacle boundary's minimal principal curvature is not small, or the speed $c$ varies little, then reflection's is strong enough to be observed.
	In addition, we require $\bn_{\bp}$ not parallel to $\bxi^{(0)}$ and $\bn_{\bp}\cdot \frac{\partial\bx}{\partial\bxi^{(0)}}\Big|_{t_{\bp}} \neq\bzero$. This condition guarantees one to find a ray starting at  $\bx^{(0)}$ with a direction in a small neighborhood of $\bxi^{(0)}$ whose reflection by the obstacle can be observed in the scattering relation (see Lemma \ref{lm:4} and \ref{lm:6}). 
	\label{as:4}
\end{enumerate}
By the differentiability theorem of initial value problems, we can easily show the following lemma.
\begin{lemma}\label{LM:CONT}
	If metric $g_{ij}=c^{-2}\delta_{ij}$ satisfies that $c(\bx)$ is $C^{k}, k\ge 3$ in $\Omega$, then the Hamiltonian system's solution $X_g$ is $C^{k-1}$ before and after hitting the obstacle. 
\end{lemma}
\begin{lemma}
	For any $\bxi\in \partial \mathcal{C}(\bx^{(0)})$, $X_g(\cdot, \bx^{(0)}, \bxi)$ is tangential to $\partial D$.
\end{lemma}
\begin{proof}
	We prove by contradiction. If $X_g(s, \bx^{(0)}, \bxi)$ impinges on the obstacle $D$ with non-tangential direction, then there is an open neighborhood $U$ at $\bxi$ such that $\forall \bxi_b\in U$, $X_g(s, \bx^{(0)}, \bxi_b)$ still intersects with $D$, which contradicts with the assumption of $\mathcal{C}(\bx^{(0)})$ contains $\mathcal{B}(\bx^{(0)})$. On the other hand, if $X_g(s, \bx^{(0)}, \bxi)$ is non-broken, then there is also an open neighborhood $V$ at $\bxi$ such that $\forall \bxi_n\in V$, $X_g(s, \bx^{(0)}, \bxi_n)$ are non-broken, which contradicts the smallness assumption of $\mathcal{C}(\bx^{(0)})$.
\end{proof}
\begin{lemma}\label{lm:3}
	If a broken ray $X_g(s, \bx^{(0)}, \bxi^{(0)})= (\bx(s), \bxi(s))$ is tangential to $\partial D$ and $\bxi^{(0)}\in \partial\mathcal{C}(\bx^{(0)})$. $U$ is an open neighborhood of $(\bx^{(0)}, \bxi^{(0)})$, if the phase $(\bx_n^{(0)}, \bxi_n^{(0)})\in U$  and the ray $X_g(s, \bx_n^{(0)}, \bxi_n^{(0)})= (\bx_n(s), \bxi_n(s))$ is non-broken with $\|(\bx^{(0)}, \bxi^{(0)}) - (\bx_n^{(0)}, \bxi_n^{(0)})\|\ll 1$, denote the traveltimes of $X_g(s, \bx^{(0)}, \bxi^{(0)})$ and $X_g(s, \bx_n^{(0)}, \bxi_n^{(0)})$ are $t$ and $t_n$ respectively, then 
	\begin{equation} 
	t_n - t = -\bu \cdot (\bx^{(0)}_n - \bx^{(0)}) -\bv\cdot  (\bxi^{(0)}_n - \bxi^{(0)})+ o(\|\bx^{(0)}_n - \bx^{(0)}\|) + o(\|\bxi^{(0)}_n - \bxi^{(0)}\|).
	\end{equation}
	where $\bu$ and $\bv$ are defined as
	\begin{equation}
	\begin{aligned}
	&\bu = \left(\nabla G\cdot \pdr{\bx}{s}\Big|_{t}\right)^{-1}\left(\nabla G\cdot \pdr{\bx}{\bx{(0)}}\Big|_{t} \right),\\
	&\bv = \left(\nabla G\cdot \pdr{\bx}{s}\Big|_{t}\right)^{-1}\left(\nabla G\cdot \pdr{\bx}{\bxi{(0)}}\Big|_{t} \right).
	\end{aligned}
	\end{equation}
\end{lemma}
\begin{proof}
	The existence of such a non-broken ray in $U$ is directly from previous lemma. At the exiting locations, $G(\bx(t)) = 0$ and $G(\bx_n(t_n)) = 0$. Since $G$ is twice differentiable, we have the following expansion,
	\begin{equation}\nonumber
	\begin{aligned}
		G(\bx_n(t_n)) =&~ G(\bx(t)) + \nabla G\cdot \pdr{\bx}{s}\Big|_{t }(t_n- t) \\&+ \nabla G\cdot \pdr{\bx}{\bx^{(0)}}\Big|_{t } (\bx^{(0)}_n - \bx^{(0)}) + \nabla G\cdot \pdr{\bx}{\bxi^{(0)}}\Big|_{t} (\bxi^{(0)}_n - \bxi^{(0)}) \\
		&+ O((t_n - t)^2) + o(\|\bx^{(0)}_n- \bx^{(0)}\|) + o(\|\bxi^{(0)}_n - \bxi^{(0)}\|).
	\end{aligned}
	\end{equation}
	According to the assumption~\ref{as:3}, there exists a constant $\eta > 0$ such that $ \Big|\nabla G\cdot \pdr{\bx}{s}\Big|_{t}\Big| >  \eta$. Therefore we can find vectors $\bu$ and $\bv$ that
	\begin{equation}
	t_n - t = -\bu \cdot (\bx^{(0)}_n - \bx^{(0)}) -\bv\cdot  (\bxi^{(0)}_n - \bxi^{(0)})+ o(\bx^{(0)}_n - \bx^{(0)}) + o(\bxi^{(0)}_n - \bxi^{(0)}),
	\end{equation}
	where 
	\begin{equation}
	\begin{aligned}
	&\bu = \left(\nabla G\cdot \pdr{\bx}{s}\Big|_{t}\right)^{-1}\left(\nabla G\cdot \pdr{\bx}{\bx^{(0)}}\Big|_{t} \right),\\
	&\bv = \left(\nabla G\cdot \pdr{\bx}{s}\Big|_{t}\right)^{-1}\left(\nabla G\cdot \pdr{\bx}{\bxi^{(0)}}\Big|_{t} \right).
	\end{aligned}
	\end{equation}
\end{proof}

\begin{lemma}\label{lm:4}
	Suppose $\bxi\in \bbS^{d-1}$ is a fixed vector and $\bv\neq \bzero$ is not parallel to $\bxi$.  For any open set $B\subset \bbS^{d-1}$ 
	$\sup_{\bzeta\in B}\|\bxi - \bzeta\| <\eps \ll 1$, 
	there exists $\hat{\bzeta}\in B$ such that 
	\begin{equation}
	|\bv\cdot (\bxi - \hat{\bzeta})| = O(\|\bxi - \hat{\bzeta}\|).
	\end{equation}
\end{lemma}
\begin{proof}
	Without loss of generality, we assume $\|\bv\| = 1$. Let $r=\bxi\cdot\bv, |r|<1$. Since $B$ is an open set in $\bbS^{d-1}$, there exists $\hat{\bzeta} \in B$, $\hat{\bzeta} \neq \bxi$ such that $|(\bxi - \hat{\bzeta})\cdot (r\bxi - \bv)| \ge \hat{\gamma}t\sqrt{1-r^2}$ for some $\hat{\gamma}>0$, where $t=\|\bxi-\hat{\bzeta}\|<\epsilon$. Since $|(\bxi - \hat{\bzeta})\cdot \bxi|=O(t^2)$, we have $|\bv\cdot (\bxi - \hat{\bzeta})| = O(\|\bxi - \hat{\bzeta}\|)$.
	
\end{proof}

\begin{lemma}\label{lm:5}
	For any initial phase $X^{(0)} = (\bx^{(0)}, \bxi^{(0)})$, suppose the ray $$X_g(s, \bx^{(0)}, \bxi^{(0)}) = (\bx(s ), \bxi(s)),\quad \bx(0)=\bx^{(0)},\quad \bxi(0)=\bxi^{(0)}$$ is non-broken on $s\in (0, T)$. If vector $\bn$ satisfies
	\begin{equation}
	\frac{\partial \bx(t)}{\partial \bxi{(0)}} \bn = \bzero,\quad 	\frac{\partial \bxi(t)}{\partial \bxi{(0)}} \bn = \bzero, \quad \forall t\in(0,T)
	\end{equation}
	then $\bn = \bzero$.
\end{lemma}
\begin{proof}
	Since the determinant of the Jacobian matrix $J = \begin{pmatrix}
	\frac{\partial \bx}{\partial \bx{(0)}} & \frac{\partial \bx}{\partial \bxi{(0)}}\\
	\frac{\partial \bxi}{\partial \bx{(0)}} & \frac{\partial \bxi}{\partial \bxi{(0)}} 
	\end{pmatrix}$ is 1, if the vector $\bn\neq \bzero$, then
	\begin{equation}
	J\begin{pmatrix}
	\bzero\\\bn
	\end{pmatrix} = \bzero
	\end{equation}
	which contradicts to the non-degeneracy of $J$.
\end{proof}
\begin{lemma}\label{lm:6}
	From a fixed point $\bx^{(0)}\in\partial\Omega$, we denote the ray with initial direction $\bzeta$ as  $X_g(s, \bx^{(0)}, \bzeta) = (\bx(s, \bzeta), \bxi(s, \bzeta))$ and let $T(\bzeta)$ be the traveltime.
 Suppose $\bxi^{(0)}\in \partial \mathcal{C}(\bx^{(0)})$, 
	then either $\partial_{\bzeta} T$ or $\partial_{\bzeta}\bx$ or $\partial_{\bzeta}\bxi$ is discontinuous at $\bxi^{(0)}$.
\end{lemma}

\begin{proof}	
	Consider a small open neighborhood $V\subset \mathcal{S}^-$ of $\bxi^{(0)}$, then the following sets 
	\begin{equation}
	\begin{aligned}
		&N = \{ \bzeta \in V~|~ X_g(\cdot, \bx^{(0)},\bzeta) \text{ is nonbroken}\},\\
		&B = \{ \bzeta\in V ~|~ X_g(\cdot, \bx^{(0)}, \bzeta) \text{ is broken}\}.
	\end{aligned}
	\end{equation}
	are both nonempty sets.
	 Then by Lemma~\ref{lm:3}, for any $\epsilon > 0$, we can select $\bxi^{(0)}_n\in N$ that $\| \bxi^{(0)}_n - \bxi^{(0)}\| < \epsilon$ and
	\begin{equation}
	\begin{aligned}
	|T(\bxi^{(0)}) - T(\bxi^{(0)}_n)| = O(\|\bxi^{(0)} - \bxi^{(0)}_n\|).
	\end{aligned}
	\end{equation}
	and the differences between the exiting locations and phases are also of order $O(\|\bxi^{(0)} - \bxi^{(0)}_n\|)$.
	\begin{figure}[!htb]
	\begin{center}
		\includegraphics[scale=0.20]{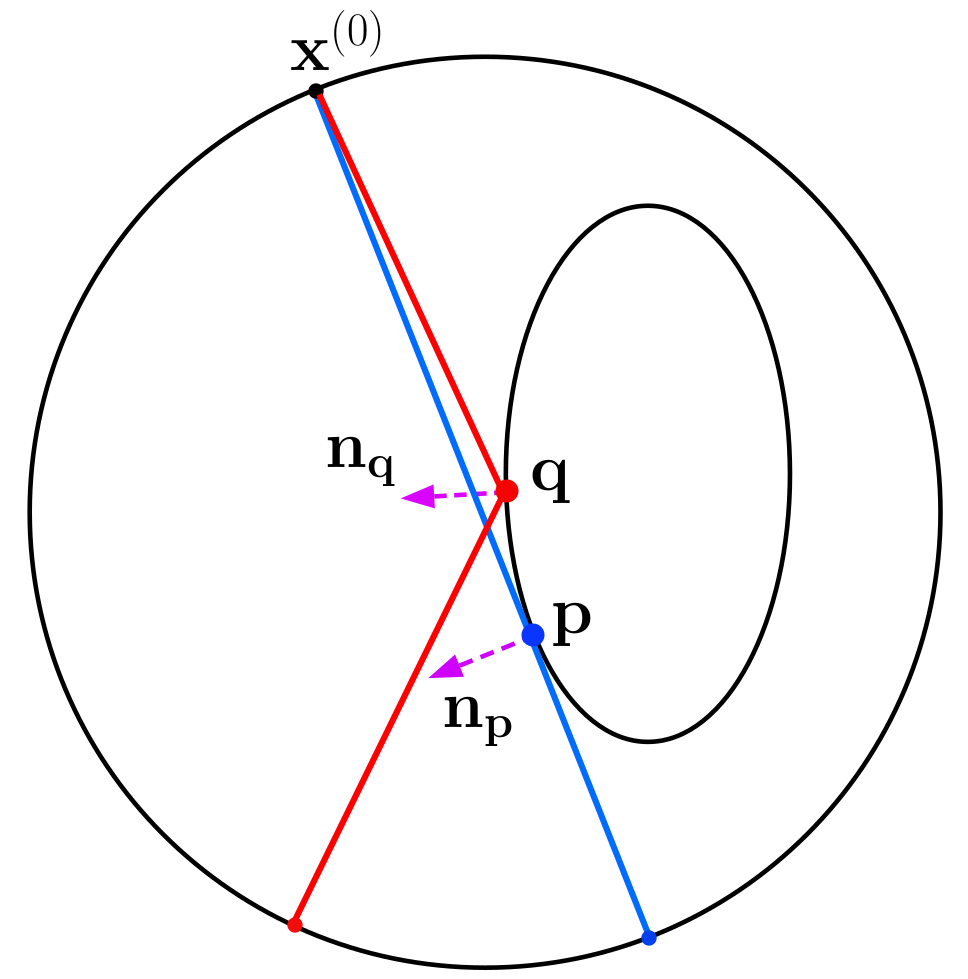}
		\caption{The illustration of the tangential ray and broken ray, both rays start from the same physical locations but different directions. The tangential ray intersects with the obstacle at $\bp$ and the other broken ray intersects with the obstacle at $\bq$.}
				\label{fig:example}
	\end{center}
	\end{figure}
	
	On the other hand, suppose the broken ray $X_g(s, \bx^{(0)}, \bxi^{(0)})$ is tangential to $\partial D$ at point $\bp = \bx(t_{\bp}, \bxi^{(0)})$, let
	$$\bz = \nabla F(\bp)\cdot \frac{\partial \bx}{\partial \bzeta}\Big|_{t_{\bp},\bxi^{(0)}},$$
	 then by the assumption~\ref{as:4}, $\bz \neq \bzero$ and not parallel to $\bxi^{(0)}$, using the Lemma~\ref{lm:4}, For the same $\epsilon$, we can select $\bxi_b^{(0)}\neq \bxi^{(0)}$ with $\bxi_b^{(0)}\in B$,  $\|\bxi_b^{(0)} - \bxi^{(0)}\|<\epsilon$ and 
	 \begin{equation}\label{eq:nonzero}
	 \bz \cdot (\bxi_b^{(0)} - \bxi^{(0)}) = O(\|\bxi_b^{(0)} - \bxi^{(0)}\|).
	 \end{equation}
	Then the ray with initial phase $X_b^{(0)} = (\bx^{(0)}, \bxi^{(0)}_b)$ is broken. 
	Note that the tangent ray $X_g(\cdot, \bx^{(0)}, \bxi^{(0)})$ satisfies 
	\begin{equation}\label{eq:tangent}
	\begin{aligned}
		\nabla F(\bp)\cdot \frac{\partial \bx}{\partial s}\Big|_{t_{\bp}, \bxi^{(0)}} = 0,\\
	F(\bp) = 0.
	\end{aligned}
	\end{equation}
	Assume the broken ray $X_g(s, \bx^{(0)}, \bxi^{(0)}_b)$ intersects $\partial D$ at $\bq = \bx(t_{\bq}, \bxi^{(0)}_b)$ as illustrated in Figure~\ref{fig:example}, then take Taylor expansion at $(t_{\bp}, \bxi^{(0)})$,
	\begin{equation}
	\begin{aligned}
	0 = F( \bx(t_{\bq}, \bxi^{(0)}_b)) =& F(\bx(t_{\bp}, \bxi^{(0)})) \\&+ \nabla F(\bp)\cdot \frac{\partial \bx}{\partial s}\Big|_{(t_{\bp},\bxi^{(0)})}(t_{\bq} - t_{\bp})+ \nabla F(\bp)\cdot \frac{\partial \bx}{\partial \bzeta}\Big|_{(t_{\bp},\bxi^{(0)})}(\bxi^{(0)}_{b} - \bxi^{(0)}) \\
	& + o(\|\bxi^{(0)}_b - \bxi^{(0)}\|) + O((t_{\bq} - t_{\bp})^2).
	\end{aligned}
	\end{equation}
	The first two terms on the right hand side are zero according~\eqref{eq:tangent}. By using~\eqref{eq:nonzero} $$\nabla F(\bp)\cdot \frac{\partial \bx}{\partial \bzeta}\Big|_{(t_{\bp},\bxi^{(0)})}(\bxi^{(0)}_{b} - \bxi^{(0)}) = \bz\cdot (\bxi^{(0)}_{b} - \bxi^{(0)}) = O(\|\bxi^{(0)}_{b} - \bxi^{(0)}\|),$$ we then conclude  
	\begin{equation}
	|t_{\bp} - t_{\bq}|= O\left(\sqrt{\|(\bxi^{(0)}_{b} - \bxi^{(0)}) \|}\right).
	\end{equation}
	After reflection at time $t_{\bq}$, the direction of the broken ray $X_g(s, \bx^{(0)}, \bxi^{(0)}_b)$ has been reflected to $$\bxi_{b,\textrm{out}} = (I - 2\bn_{\bq}\otimes \bn_{\bq})\bxi(t_{\bq}^{-}, \bxi^{(0)}_b),$$
	where $\bn_{\bq} = \frac{\nabla F(\bq)}{\|\nabla F(\bq)\|}$ is the outward unit normal vector at $\bq$. Then by taking Taylor expansion at $(t_{\bp}, \bxi^{(0)})$, the difference between the directions of the two rays $X_g(t_{\bq}^+,\bx^{(0)}, \bxi^{(0)})$ and $X_g(t_{\bq}^+,\bx^{(0)}, \bxi^{(0)}_b)$ is
	\begin{equation}
	\begin{aligned}
		\bxi(t_{\bq}, \bxi^{(0)}) - \bxi_{b,\textrm{out}}&= 2 \bn_{\bp} \left(\bn_{\bp}\cdot \frac{\partial \bxi}{\partial s}\Big|_{t_{\bp},\bxi^{(0)}} +\left( \frac{\partial \bx}{\partial s}\nabla\frac{\nabla F}{\|\nabla F\| }\bxi\right)\Big|_{t_{\bp},\bxi^{(0)}}\right) (t_{\bq} - t_{\bp}) \\&\quad + O((t_{\bq} - t_{\bp})^2).
	\end{aligned}
	\end{equation}
	where $\bn_{\bp} = \frac{\nabla F(\bp)}{\|\nabla F(\bp)\|}$ is the outward normal vector at $\bp$. According to assumption~\ref{as:4}, the first term on right hand side does not vanish, therefore
	\begin{equation}\label{eq:k1}
	\bxi(t_{\bq}, \bxi^{(0)}) - \bxi_{b,\textrm{out}} = \bn_{\bp} \cdot O(t_{\bq} - t_{\bp}) = \bn_{p}\cdot  O\left(\sqrt{\|(\bxi^{(0)}_{b} - \bxi^{(0)}) \|}\right).
	\end{equation}
	Regarding the phases at time $t_{\bq}^{+}$ as initial phases, then we can define the rest of the tangential ray as 
	\begin{equation}
	X_g(s, \bx(t_{\bq}, \bxi^{(0)}), \bxi(t_{\bq}, \bxi^{(0)})) = (\by(s), \btheta(s)), \quad \by(0)=\bx(t_{\bq}, \bxi^{(0)}),\quad \btheta(0)=\bxi(t_{\bq}, \bxi^{(0)}),
	\end{equation}
	and denote the traveltime of this partial ray as $t_{\by}$. For the other ray, we define the rest of the broken ray as
	\begin{equation}
	X_g(s, \bq, \bxi_{b,\textrm{out}}) = (\bw(s), \bEta(s)), \quad \bw(0) = \bq,\quad \bEta(0) = \bxi_{b, \textrm{out}},
	\end{equation}
	and similarly, the traveltime is denoted as $t_{\bw}$. 
	Then use the Lemma~\ref{lm:3} and~\eqref{eq:k1} 
	\begin{equation}
	\begin{aligned}
	T(\bxi^{(0)}) - T(\bxi^{(0)}_b) &= (t_{\bq} + t_{\by}) - (t_{\bq} + t_{\bw})\\&= -\bu \cdot (\by(0) - \bw(0)) -\bv \cdot (\btheta(0) - \bEta(0)) \\&\quad+ o(\|\by(0) - \bw(0)\|)+ o(\|\btheta(0) - \bEta(0)\|) \\
	&= -(\bv\cdot \bn_{p}) \cdot O\left(\sqrt{\|(\bxi^{(0)}_{b} - \bxi^{(0)}) \|}\right) + O(\|(\bxi^{(0)}_{b} - \bxi^{(0)}) \|),
	\end{aligned}
	\end{equation}
	where $\bu$ and $\bv$ are defined as following,
	\begin{equation}
	\begin{aligned}
	\bu = \left(\nabla G\cdot \frac{\partial \by}{\partial s}\Big|_{t_{\by}}\right)^{-1} \left(\nabla G\cdot \frac{\partial {\by}}{\partial \by(0)}\Big|_{t_{\by}}\right),\\
	\bv = \left(\nabla G\cdot \frac{\partial \by}{\partial s}\Big|_{t_{\by}}\right)^{-1} \left(\nabla G\cdot \frac{\partial {\by}}{\partial \btheta(0)}\Big|_{t_{\by}}\right).
	\end{aligned}
	\end{equation}
	We consider following three cases,
	\begin{enumerate}
		\item If $\bv\cdot \bn_p\neq 0$, then $\partial_{\zeta}T$ is discontinuous at $\bxi^{(0)}$. Otherwise, $T(\bxi^{(0)}) - T(\bxi^{(0)}_b) =t_{\by} - t_{\bw}= O(\|(\bxi^{(0)}_{b} - \bxi^{(0)}) \|)$ and we consider the next case.
		\item Use~\eqref{eq:k1} and $\|\bx(t_{\bq}, \bxi^{(0)}) - \bq\|=O(\|(\bxi^{(0)}_{b} - \bxi^{(0)}) \|)$, the difference between exiting physical locations is
		\begin{equation}
		\begin{aligned}
		\by(t_{\by}) - \bw(t_{\bw}) &= \frac{\partial \by}{\partial s}\Big|_{t_{\by}} (t_{\by} - t_{\bw}) \\&\quad+\frac{\partial \by}{\partial \by(0)}\Big|_{t_{\by}}(\bx(t_{\bq}, \bxi^{(0)}) - \bq) +\frac{\partial \by}{\partial \btheta(0)}\Big|_{t_{\by}}(\bxi(t_{\bq}, \bxi^{(0)}) - \bxi_{b,\textrm{out}})\\&\quad+O ((t_{\by} - t_{\bw})^2)+O(\|\bx(t_{\bq}, \bxi^{(0)})-\bq\|^2) + O(\|(\bxi(t_{\bq}, \bxi^{(0)}) - \bxi_{b,\textrm{out}}\|^2)\\
		&= \left( \frac{\partial \by}{\partial \btheta(0)}\Big|_{t_{\by}}\cdot\bn_p \right)\cdot  O\left(\sqrt{\|(\bxi^{(0)}_{b} - \bxi^{(0)}) \|}\right) + O(\|(\bxi^{(0)}_{b} - \bxi^{(0)}) \|).
		\end{aligned}
		\end{equation}
		If $\frac{\partial \by}{\partial \btheta(0)}\Big|_{t_{\by}}\cdot\bn_p $ is nonzero vector, then the derivative $\partial_{\bzeta}\bx$ will suffer from a discontinuity at $\by(t_{\by})$, otherwise if $\frac{\partial \by}{\partial \btheta(0)}\Big|_{t_{\by}}\cdot\bn_p = \bzero$, we consider the next case.
		\item Similarly, the difference between exiting directions is
		\begin{equation}
		\begin{aligned}
		\btheta(t_{\by}) - \bEta(t_{\bw}) &= \frac{\partial \btheta}{\partial s}\Big|_{t_{\by}}(t_{\by} - t_{\bw})\\&\quad +\frac{\partial \btheta}{\partial \by(0)}\Big|_{t_{\by}}(\bx(t_{\bq}, \bxi^{(0)}) - \bq) +  \frac{\partial \btheta}{\partial\btheta(0)}\Big|_{t_{\by}}(\bxi(t_{\bq}, \bxi^{(0)}) - \bxi_{b,\textrm{out}})\\&\quad + O ((t_{\by} - t_{\bw})^2)+O(\|\bx(t_{\bq}, \bxi^{(0)})-\bq\|^2) + O(\|(\bxi(t_{\bq}, \bxi^{(0)}) - \bxi_{b,\textrm{out}}\|^2)\\
		&= \left( \frac{\partial \btheta}{\partial \btheta(0)}\Big|_{t_{\by}}\cdot\bn_p \right)\cdot  O\left(\sqrt{\|(\bxi^{(0)}_{b} - \bxi^{(0)}) \|}\right) + O(\|(\bxi^{(0)}_{b} - \bxi^{(0)}) \|).
		\end{aligned}
		\end{equation}
		If $ \frac{\partial \by}{\partial \btheta(0)}\Big|_{t_{\by}}\cdot\bn_p$ is nonzero vector, then the derivative $\partial_{\bzeta}\bxi$ will suffer from a discontinuity at $\btheta(t_{\by})$. Otherwise we will have following equations
		\begin{equation}
		\begin{aligned}
		\frac{\partial \by}{\partial \btheta(0)}\Big|_{t_{\by}}\cdot\bn_p &= 0, \\
		\frac{\partial \btheta}{\partial \btheta(0)}\Big|_{t_{\by}}\cdot\bn_p &= 0,
		\end{aligned}
		\end{equation}
		by Lemma~\ref{lm:5}, we must have $\bn_{\bp}=\bzero$, which is a contradiction. Therefore either $\partial_{\bzeta}T$ or $\partial_{\bzeta}\bx$ or $\partial_{\bzeta}\bxi$ must have a discontinuity at $\bxi^{(0)}$.
	\end{enumerate}
\end{proof}
With these assumptions, we can directly detect non-broken rays from the measurements by scanning the traveltimes, exiting directions and exiting locations for jumps in the derivatives with respect to initial directions, see Figure~\ref{fig:jumps}. And using these non-broken rays enables us to recover the metric outside of the convex hull of the obstacle~\cite{krishnan2009support,ilmavirta2014broken}.
\begin{figure}[!htb]
	\centering
	\includegraphics[scale=0.14]{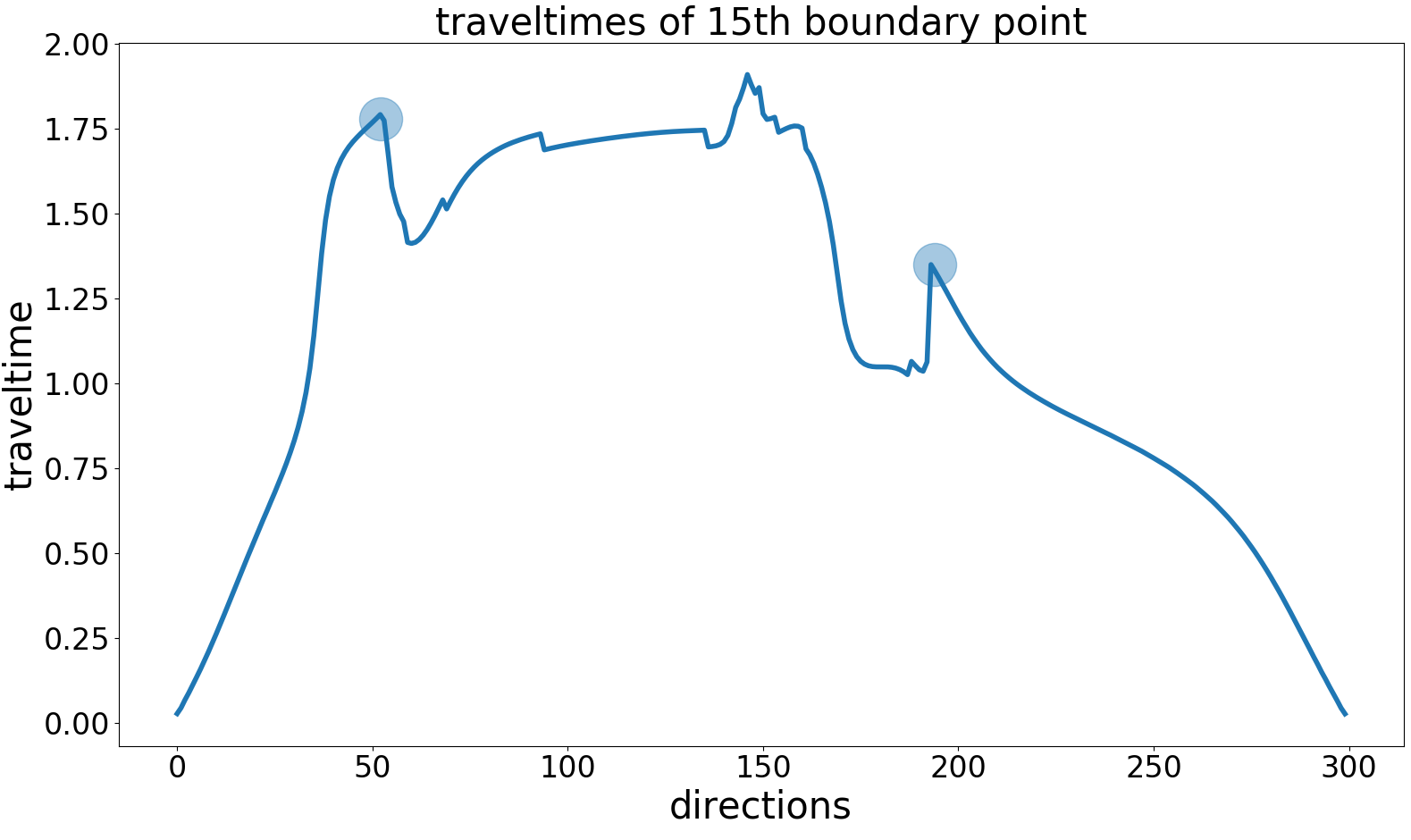}
	\includegraphics[scale=0.14]{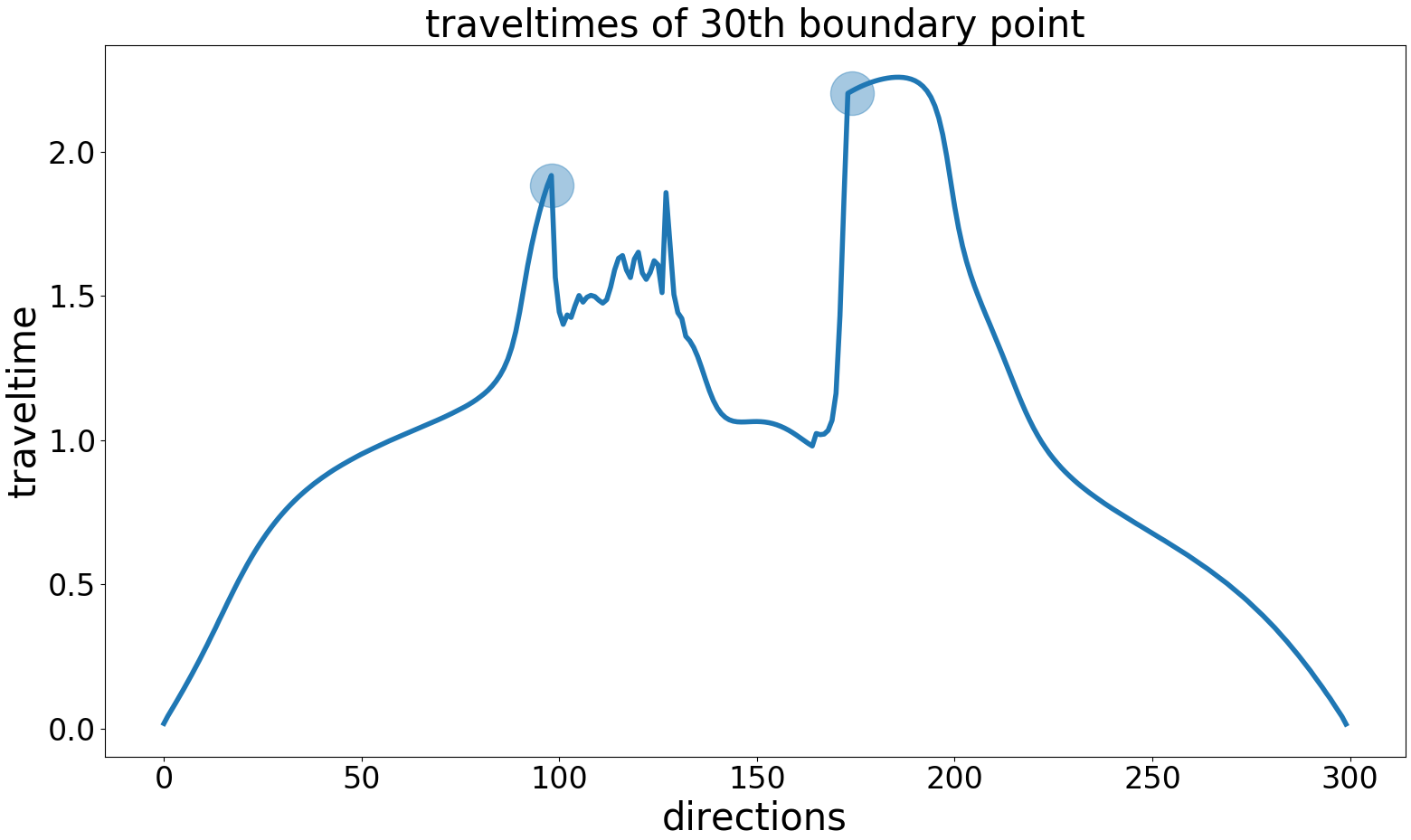}\\
		\includegraphics[scale=0.14]{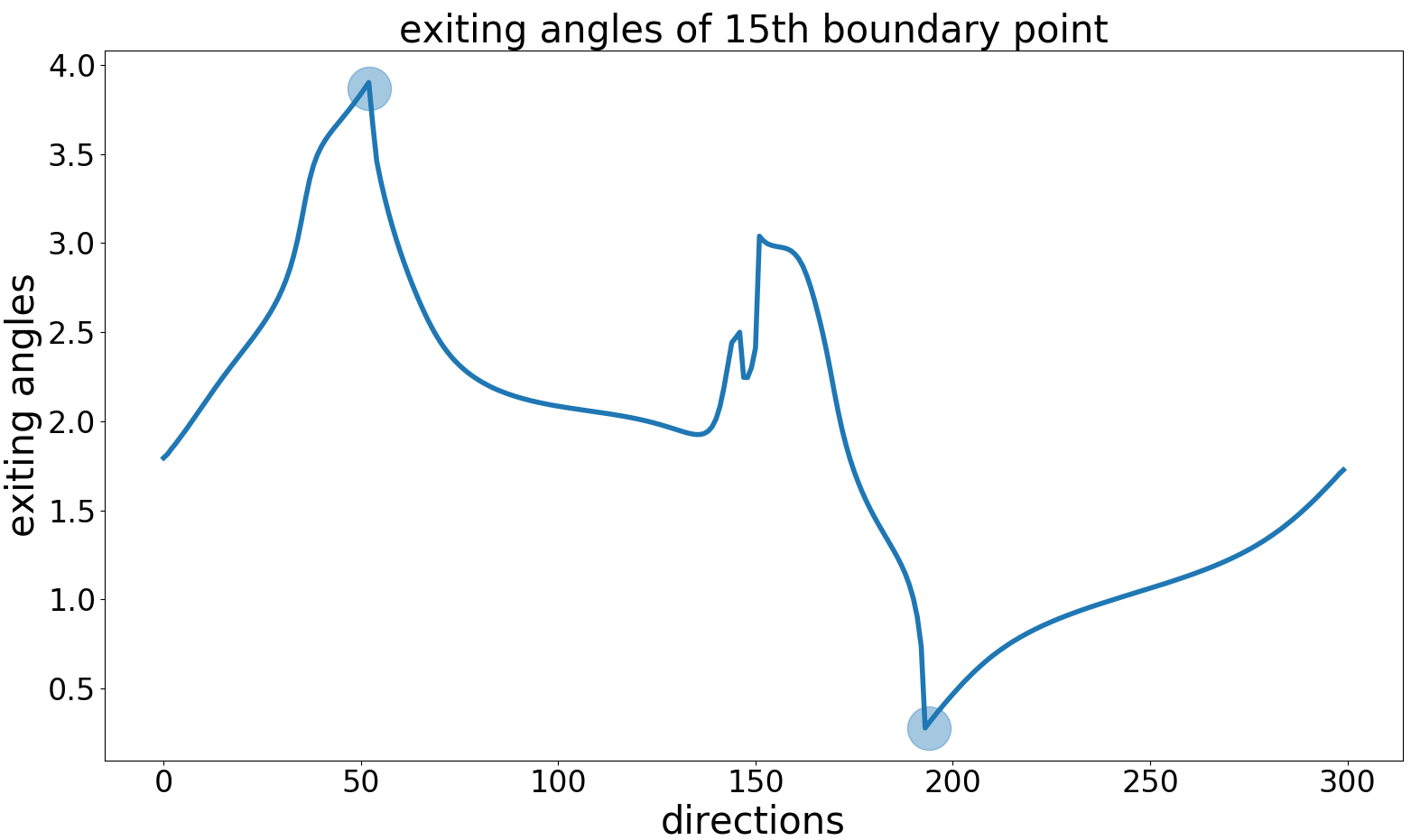}
		\includegraphics[scale=0.14]{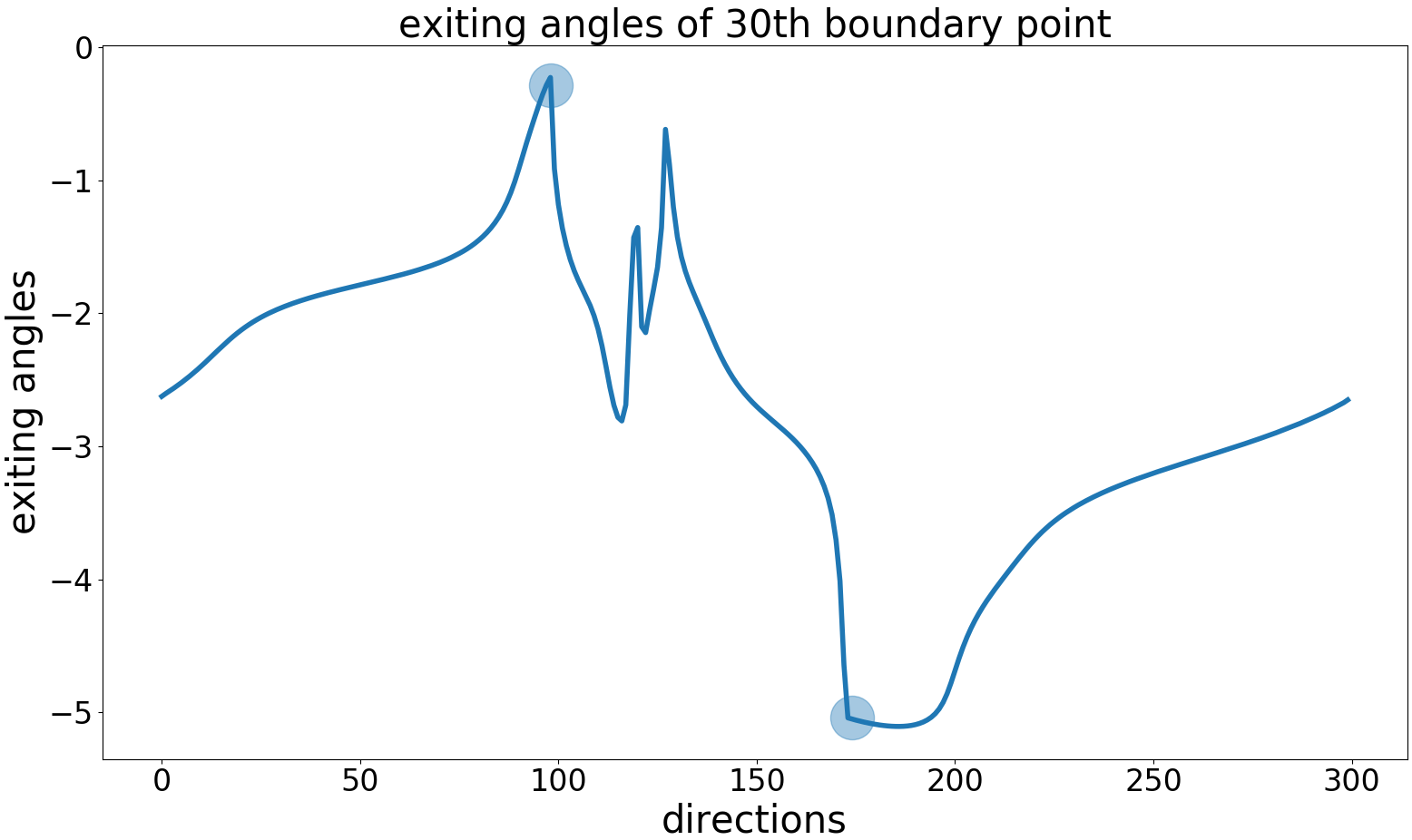}\\
			\includegraphics[scale=0.14]{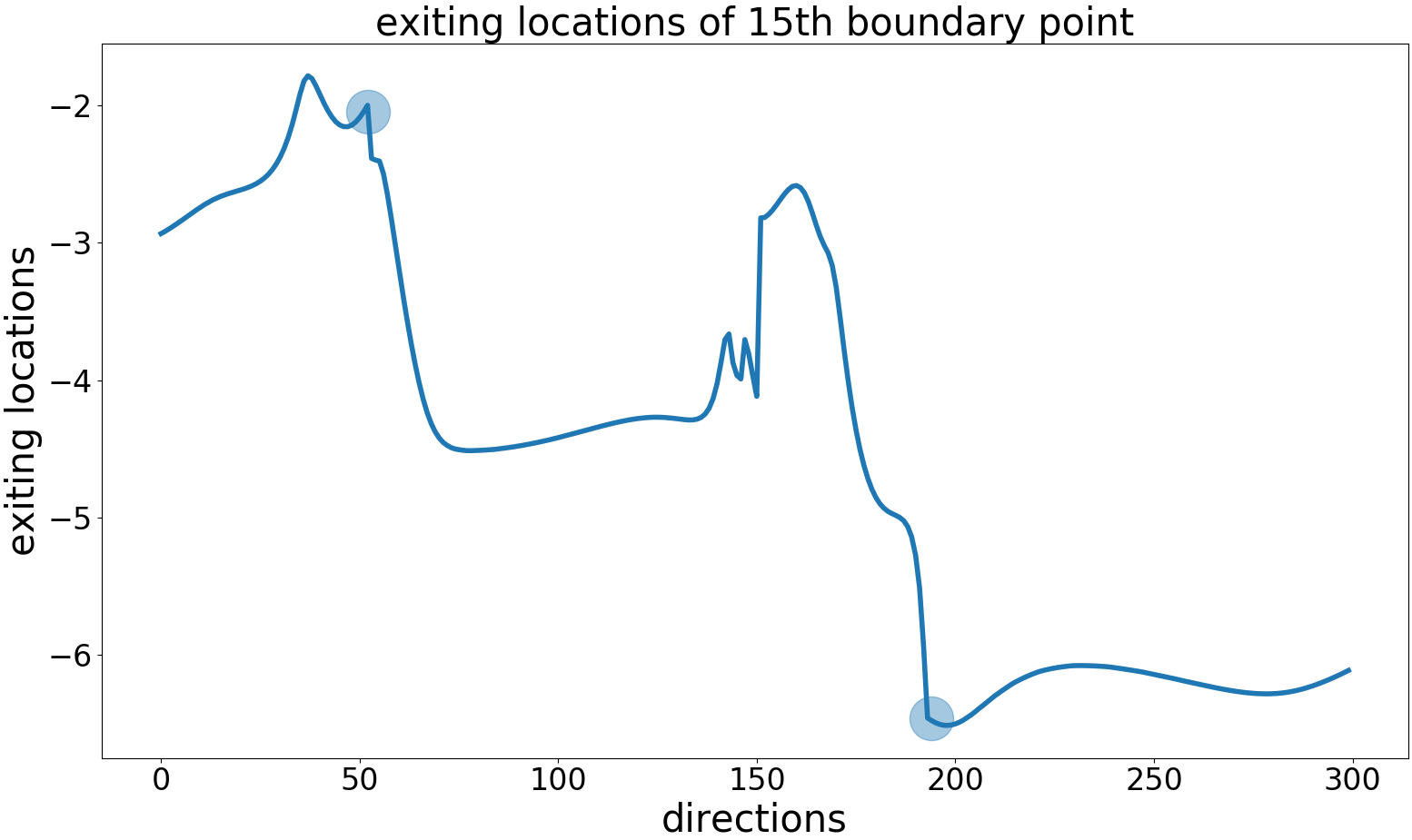}
			\includegraphics[scale=0.14]{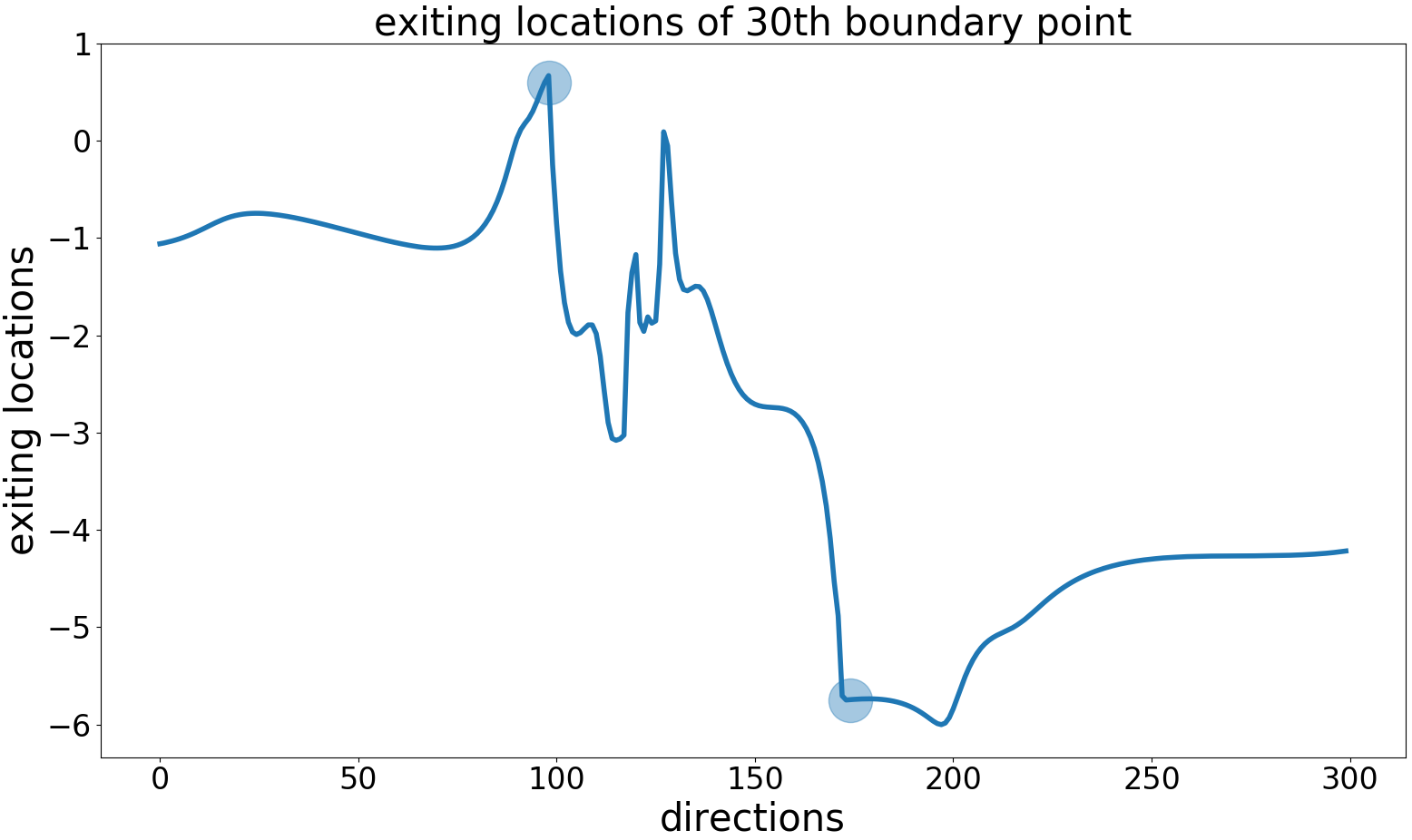}
	\caption{The plots of traveltimes, exiting directions, exiting locations from numerical example~\ref{num:5}. Left: From top to bottom, the plots are the traveltimes, exiting directions, exiting locations corresponding to the initial phases with varying directions at the \nth{15} boundary point, the shadowed spots are placed at the detected jumps near \nth{52} and \nth{195} rays respectively. Right: From top to bottom, the plots are the traveltimes, exiting directions, exiting locations corresponding to the initial phases with varying directions at the \nth{30} boundary point, the shadowed spots are placed at the detected jumps near \nth{98} and \nth{175} rays respectively.}
	\label{fig:jumps}
\end{figure}

\subsection{Reconstruction of metric and obstacle}
In this section, we present the hybrid method for reconstructing both the metric and included obstacles.  First we find out if there exists an obstacle inside the medium by checking if the set $T$ defined in \eqref{eq:broken} is empty. 

If no obstacle is detected, then we can use the improved adaptive phase space method in Section~\ref{sec:stab-adaptive} to recover the metric efficiently and stably with the layer-stripping strategy. 

Once an obstacle is detected, then we distinguish the broken rays and non-broken ones by scanning the exiting traveltimes, locations and directions in scattering relation explained in Section~\ref{sec:non-broken-detect} and use the improved adaptive phase space method in Section~\ref{sec:stab-adaptive} on non-broken rays to recover the metric. In this case, our layer-stripping reconstruction strategy will be able to recover the metric starting from the boundary and continuing inward all the way to the convex hull of the obstacle. Since our adaptive phase space method is based on optimization formulation with a regularization  \eqref{eq:regularize} for the metric at all grid points in the domain, the numerically reconstructed metric in the whole domain can be viewed as a good approximation of the true metric outside the convex hull of the obstacle plus a harmonic extension to the interior of the convex hull.
	
If the obstacle is convex (under the metric), one can reconstruct both the metric and the obstacle using non-broken rays as we can see from the numerical experiments in Section~\ref{sec:num}. Since non-broken rays only contain information outside the convex hull of the obstacle, it is impossible to reconstruct a concave obstacle with only non-broken rays.
	
On the other hand, since our adaptive phase method using non-broken rays reconstructs the metric on the whole domain, then by tracing back the rays in  $T$ (defined in \eqref{eq:broken}) to half of the traveltime, we will get the approximated reflection points of such rays, if the true metric varies slowly inside the convex hull. This gives us a direct imaging method for the boundary of the obstacle, see the numerical experiments in Section~\ref{sec:num}. 

For the cases that the metric has \emph{large} variations inside the convex hull, our method then can not recover the obstacle and metric inside the convex hull without using other broken rays. One possible way is to introduce an representation of the obstacle's boundary and iteratively  reconstruct the metric inside the convex hull as well as morph the boundary simultaneously to minimize the mismatch, e.g., using the result from our hybrid method as an initial guess. However, this will be a daunting task due to the highly non-convex and coupled optimization problem.
	
Finally, we briefly discuss the computational cost of our hybrid method in 2D. Suppose we have $N_s$ sources and each source probes $N_a$ directions, then there are $N_sN_a$ scattering relation measurements. For obstacle detection, it will take $O(N_s N_a)$ complexity at the worst case. For non-broken rays detection, it will take $O(N_s N_a)$ complexity due to linear scan. At each iteration of the stabilized adaptive phase space method, we have to solve the Hamiltonian system for $X_g(s, X^{(0)})$ in~\eqref{eq:hamiltonian} and Jacobian matrix~\eqref{eq:jacobian}, which has the worst complexity as $O(T_{\max}N_s N_a)$, where $T_{\max}$ is length of the longest geodesic. These solutions are then used to calculate mismatch and formulate the linearized Stefanov-Ulhmann identity~\eqref{eq:linearized-su} over the Eulerian grid in~\eqref{eq:linear-eq} with the worst complexity $O(T_{\max}N_sN_a)$. Then we use the standard multifrontal solver \texttt{umfpack} to solve the perturbation for the minimization problem~\eqref{eq:regularize}, the complexity is $O(n^3)$ in general, where $n$ is the number of unknowns. Therefore the total time complexity is $O(K( T_{\max}N_sN_a  + n^3))$, where $K$ is the number of iterations. 
\input{numerical}
\section{Conclusion}
In this work, we proposed a hybrid phase space method for traveltime tomography which includes both an unknown medium and unknown scatterer. The underlying medium outside the convex hull of the scatterer is reconstructed by a optimization based iterative method. The newly developed method is more stable than the previous adaptive phase method proposed in  \cite{chung2011adaptive} due to the introduction of an auxiliary fidelity function to guide the layer stripping process and a direct detection of all non-broken rays. To image the boundary of the scatterer, we use a direct imaging method that can locate points on the boundary of the scatterer by selecting those broken-once rays that hit the scatterer almost normally and  tracing back those 
rays to half traveltime in the reconstructed medium. 
\section*{Acknowledgement}
H. Zhao is partially supported by NSF grant DMS-1418422. Both authors would like to thank ICERM 2017 Fall program on Mathematical and Computational Challenges in Radar and Seismic Reconstruction, where this project was started. The authors also would like to thank Kui Ren for valuable discussions.
\bibliographystyle{unsrtnat}
\bibliography{main}
	
 \end{document}

%% file: numerical.tex
\section{Numerical experiments}\label{sec:num}
All numerical experiments are implemented in \texttt{Julia} and performed on a dual-core laptop of $2.7\texttt{GHz}$ CPU and $16\texttt{GByte}$ memory. Source code is hosted on \href{https://github.com/lowrank/ray}{https://github.com/lowrank/ray}.

We take the physical domain $\Omega$ as unit disk for all examples. The discretization of  metric $g$ over a uniform grid is parametrized by $\texttt{Q4}$ element. If a ray passes through a grid, then it will involve $12$ surrounding grid values, under such situation we can set rank threshold $r_{\min} = 12$. For other parameters, our selections are conservative, we take $\tau = 5\%$ and $\alpha = 10$ for fidelity function updating, and regularization parameter $\beta = 0.5$ (see Section \ref{sec:stab-adaptive}), the numerical tolerance $\epsilon=0.5\%$ for obstacle detection (see Section \ref{sec:ref-detect}). We keep them fixed for all of the examples. 

\subsection{Scenario 1: no obstacle}
In this scenario, we experiment our improved adaptive phase space method  described in Section \ref{sec:stab-adaptive} on simple cases without interior obstacle.
\subsubsection{Example 1}
The exact solution is $$c(x,y) = 1+0.3\sin(\pi x)\sin(\pi y).$$ The grid's resolution is $h = 1/15$. We put $50$ equispaced sources and each source probes $100$ uniformly distributed directions. The method converges to a solution with relative $L^2$ error $2.41\times10^{-3}$ at $11$th iteration. We plot the numerical and the exact solutions in Figure~\ref{fig:ex1}.
\begin{figure}[!htb]
	\centering
	\includegraphics[scale=0.35]{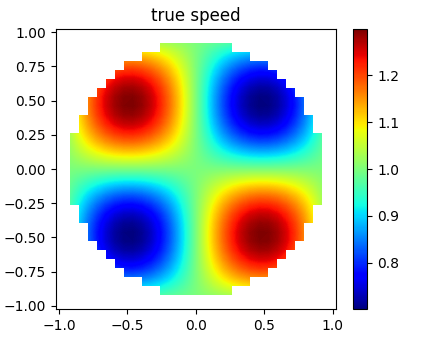}
	\includegraphics[scale=0.35]{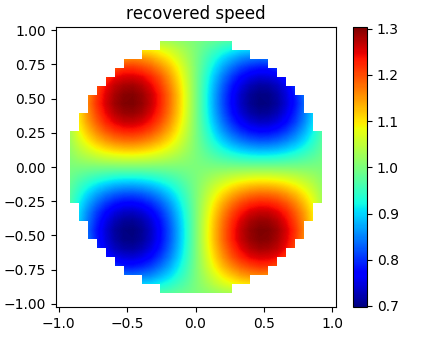}
	\includegraphics[scale=0.35]{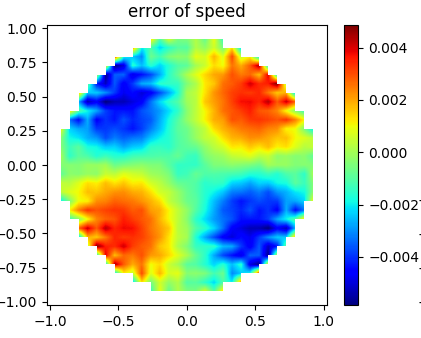}
	\caption{Left: the exact solution. Middle: the numerical solution at $11$th iteration. Right: the error between numerical solution and exact solution.}
	\label{fig:ex1}
\end{figure}
\subsubsection{Example 2}
The exact solution is $$c(x,y) = 1+0.3\sin(1.5\pi x)\sin(1.5\pi y).$$
The grid's resolution is $h = 1/25$. We put $100$ equispaced sources and each source probes $100$ uniformly distributed directions. The method converges to a solution with relative $L^2$ error $3.07\times 10^{-3}$ at $22$th iteration. We plot the numerical and the exact solutions in Figure~\ref{fig:ex2}.
\begin{figure}[!htb]
	\centering
	\includegraphics[scale=0.35]{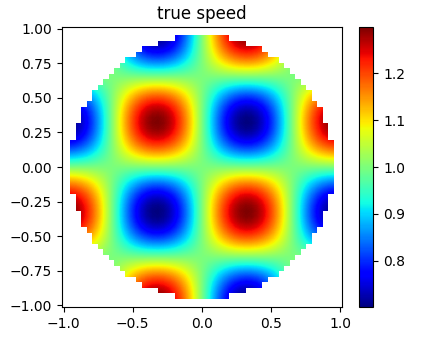}
	\includegraphics[scale=0.35]{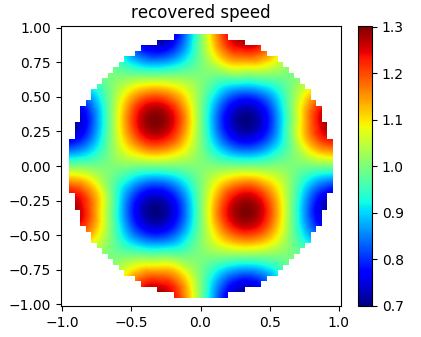}
	\includegraphics[scale=0.35]{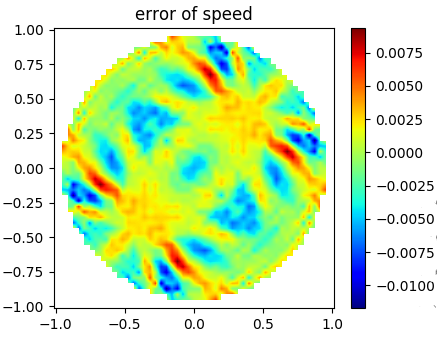}
	\caption{Left: the exact solution. Middle: the numerical solution at $22$th iteration. Right: the error between numerical solution and exact solution.}
	\label{fig:ex2}
\end{figure}
We also plot the auxiliary fidelity function at three different iterations to illustrate the layer stripping process in Figure~\ref{fig:ex2-fidelity}.
\begin{figure}[!htb]
	\centering
	\includegraphics[scale=0.38]{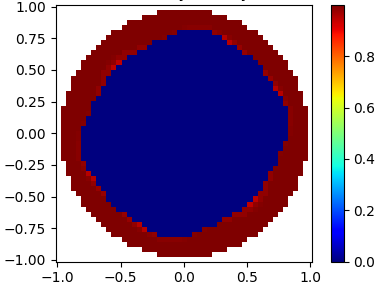}
	\includegraphics[scale=0.38]{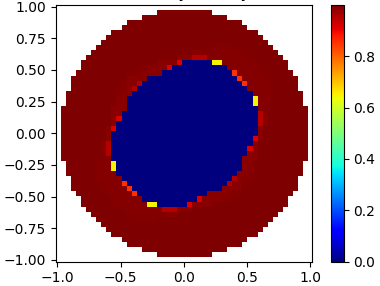}
	\includegraphics[scale=0.38]{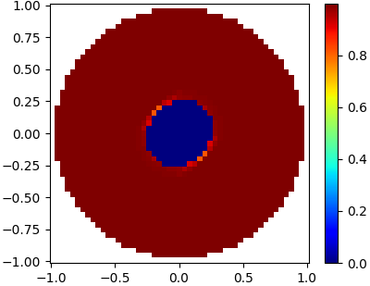}
	\caption{The fidelity function $p^n$ in different iterations. From left to right: $6$th, $11$th, $17$th iteration.}
	\label{fig:ex2-fidelity}
\end{figure}

\subsection{Scenario 2: convex unknown obstacle}
In this scenario, we experiment our hybrid method for imaging an unknown convex obstacle inside an unknown metric. 
\subsubsection{Example 3}\label{sec:ex3}
We consider the obstacle as a circle at center
$$x^2 + y^2 = \frac{1}{16},$$
and the exact solution is given by
$$c(x,y) = 1+0.4\sin\left(\pi \sqrt{(x-0.5)^2 + (y-0.2)^2}\right) + 0.4 \sin\left(\pi \sqrt{(x+0.4)^2 + (y+0.3)^2}\right).$$
The grid's resolution is $1/15$. We put $50$ equispaced sources and each source probes $300$ uniformly distributed directions. We first detect the obstacle by checking the scattering relation as in Section~\ref{sec:ref-detect}, and then distinguish the non-broken rays from the broken rays as in Section~\ref{sec:non-broken-detect}, see Figure~\ref{fig:ex3-ortho}.
\begin{figure}[!htb]
	\centering
	\includegraphics[scale=0.35]{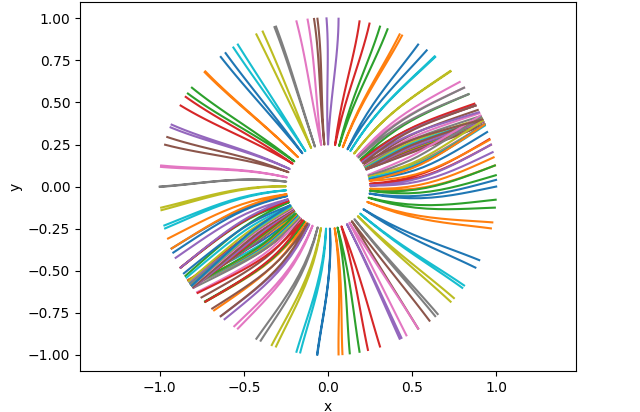}
	\includegraphics[scale=0.35]{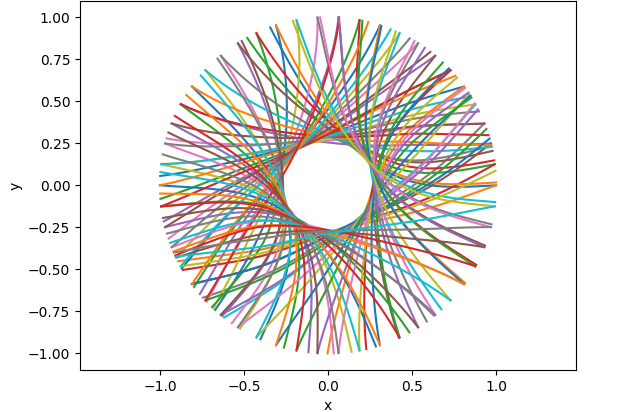}
	\caption{Left: The broken rays in $T$, which hit the obstacle in (nearly) normal direction in Example 3. Right: Detected tangent rays in Example 3.}
	\label{fig:ex3-ortho}
\end{figure}
And then we use all the non-broken rays to reconstruct the metric by the stabilized adaptive phase space method, the method converges to a solution with relative $L^2$ error $4.93\times 10^{-3}$ at $9$th iteration. We plot the numerical and exact solutions in Figure~\ref{fig:ex3}. 
\begin{figure}[!htb]
	\includegraphics[scale=0.35]{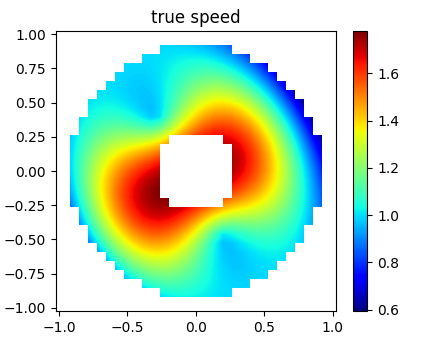}
	\includegraphics[scale=0.35]{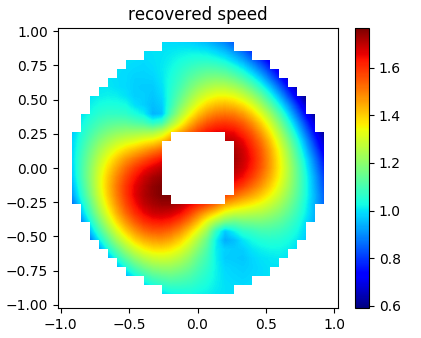}
	\includegraphics[scale=0.35]{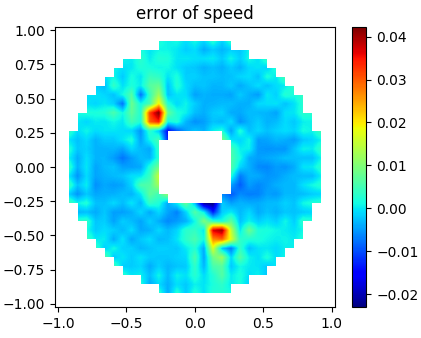}
	\caption{Left: exact solution. Middle: numerical solution at $9$th iteration. Right: error between numerical solution and exact solution.}
	\label{fig:ex3}
\end{figure}
From the experiment, we can see that the metric outside of the obstacle has been recovered well. Then the obstacle's convex hull can be approximated by the envelope of all the tangent rays computed through the recovered metric, see Figure~\ref{fig:ex3-convex-hull-rec}. After the reconstruction of the convex hull of the obstacle and the metric outside the convex hull, we trace the rays in collection $T$ (defined in \eqref{eq:broken}) to half of the traveltime to get the reflection points on the boundary, see also in Figure~\ref{fig:ex3-convex-hull-rec}. We can see that the computed reflection points are quite close to the boundary.

\begin{figure}
	\centering
	\includegraphics[scale=0.35]{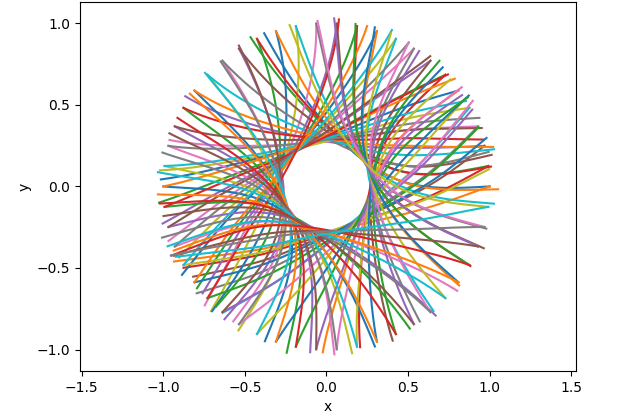}
	\includegraphics[scale=0.35]{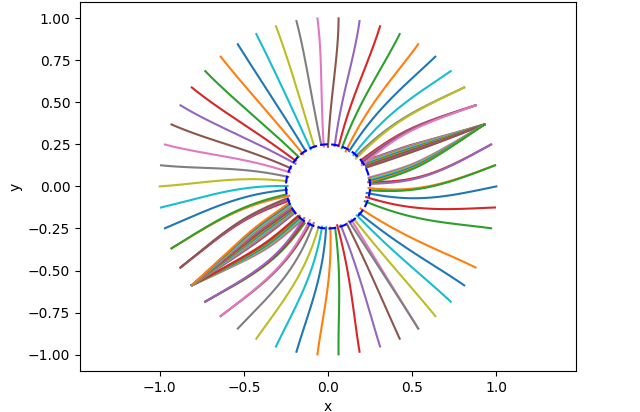}
	\caption{Left: The tangent rays computed from recovered metric in Example 3. Right: Traced the rays in collection $T$ to half traveltime, the dashed blue circle at center is the exact boundary.}
	\label{fig:ex3-convex-hull-rec}
\end{figure}

\subsection{Scenario 3: non-convex unknown obstacle}
In this scenario, we will use our hybrid method to recover a non-convex unknown interior obstacle and the underlying metric.
\subsubsection{Example 4}
In this example, we consider an easier case. The obstacle's boundary is parameterized in polar coordinate $(r, \theta)$ as
$$r(\theta) = 0.25 + 0.05\sin(3\theta),$$
which is a \emph{slightly} concave shape, the exact solution is given by
$$c(x,y) = 1+0.4\sin\left(\pi \sqrt{(x-0.5)^2 + (y-0.2)^2}\right) + 0.4 \sin\left(\pi \sqrt{(x+0.4)^2 + (y+0.3)^2}\right).$$
The grid's resolution is $1/15$. We put $50$ equispaced sources and each source probes $300$ uniformly distributed directions.  From the scattering relation, we can directly extract the rays that hit the obstacle in almost normal direction, and also distinguish the non-broken rays by detecting the jumps in scattering relation. We plot those rays in Figure~\ref{fig:ex4-shape}.
\begin{figure}[!htb]
	\centering
	\includegraphics[scale=0.35]{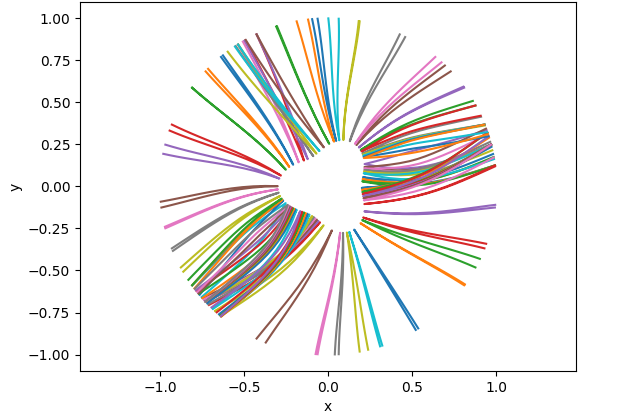}
	\includegraphics[scale=0.35]{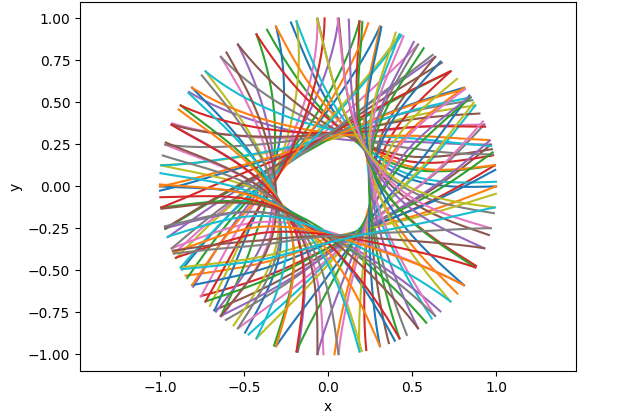}
	\caption{Left: The rays that hit the obstacle in (nearly) normal direction in Example 4. Right: The tangent rays detected from scattering relation.}
	\label{fig:ex4-shape}
\end{figure}
Then we follow the method in Section~\ref{sec:non-broken-detect} to distinguish the  non-broken rays and broken rays. Then we use all the non-broken rays to recover the metric outside the obstacle by the stabilized adaptive phase space method. The method converges to a solution with relative $L^2$ error $5.58\times 10^{-3}$ at $9$th iteration. We plot the numerical and exact solutions in Figure~\ref{fig:ex4}. 

\begin{figure}[!htb]
	\includegraphics[scale=0.35]{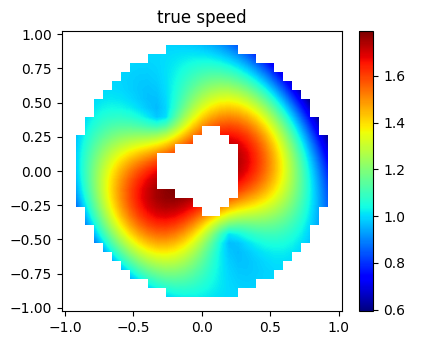}
	\includegraphics[scale=0.35]{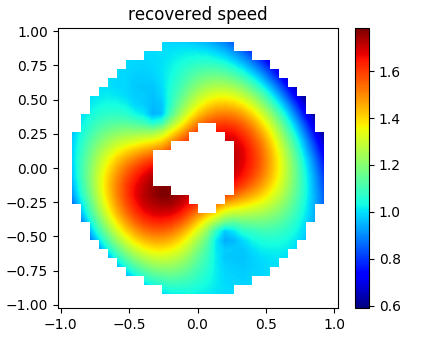}
	\includegraphics[scale=0.35]{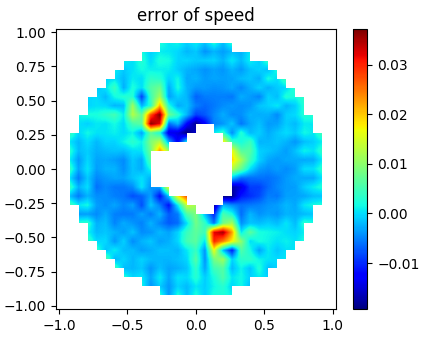}
	\caption{Left: exact solution. Middle: numerical solution at $9$th iteration. Right: error between the numerical solution and the exact solution.}
	\label{fig:ex4}
\end{figure}
Since the error of metric is small outside the obstacle, then the obstacle's convex hull can be approximated well by all the tangent rays computed through the recovered metric, see Figure~\ref{fig:ex4-nonbroken-rec}. By tracing back the rays in $T$, we approximately obtain the reflection points on the obstacle, also see Figure~\ref{fig:ex4-nonbroken-rec}. However, since no information is available inside the convex hull, the error of reflection points can be large in general.
\begin{figure}
	\centering
	\includegraphics[scale=0.35]{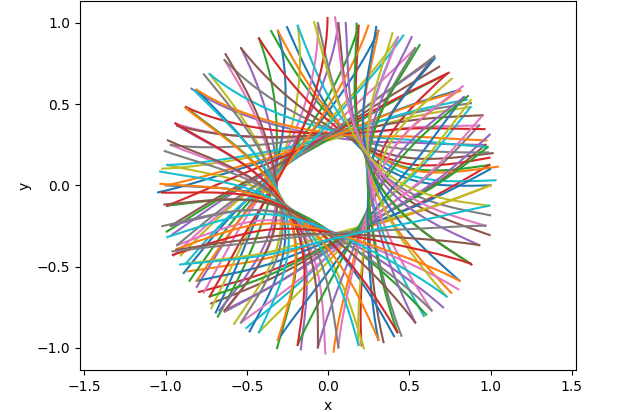}
	\includegraphics[scale=0.35]{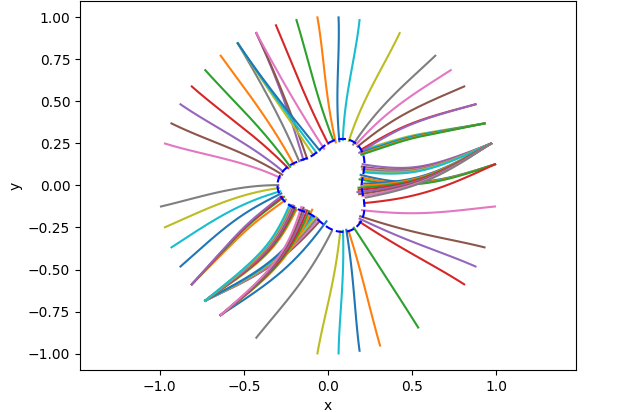}
	\caption{Left: The tangent rays computed from recovered metric in Example 4. Right: The traced rays in collection $T$ to half of traveltime. The blue dashed line is the exact boundary of obstacle.}
	\label{fig:ex4-nonbroken-rec}
\end{figure}
\subsubsection{Example 5}\label{num:5}
In this example, we take a more challenging obstacle. The obstacle's boundary is parameterized in polar coordinate $(r, \theta)$ as
$$r(\theta) = 0.4 + 0.2\sin(3\theta),$$
which is a \emph{more} concave shape than previous example,
and the exact solution is again given by
 $$c(x,y) = 1+0.4\sin\left(\pi \sqrt{(x-0.5)^2 + (y-0.2)^2}\right) + 0.4 \sin\left(\pi \sqrt{(x+0.4)^2 + (y+0.3)^2}\right).$$
The grid's resolution is $1/15$. We put $50$ equispaced sources and each source probes $300$ uniformly distributed directions.  From the scattering relation, we can directly extract the rays that hit the obstacle in almost normal direction, and also distinguish the non-broken rays by detecting the jumps in scattering relation. We plot such rays in Figure~\ref{fig:ex5-rays}.
\begin{figure}[!htb]
	\centering
	\includegraphics[scale=0.35]{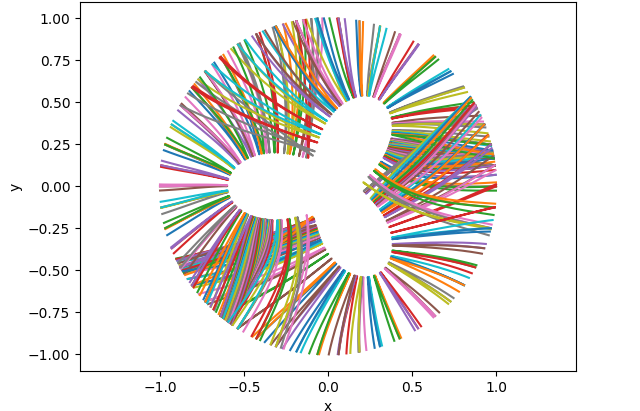}
	\includegraphics[scale=0.35]{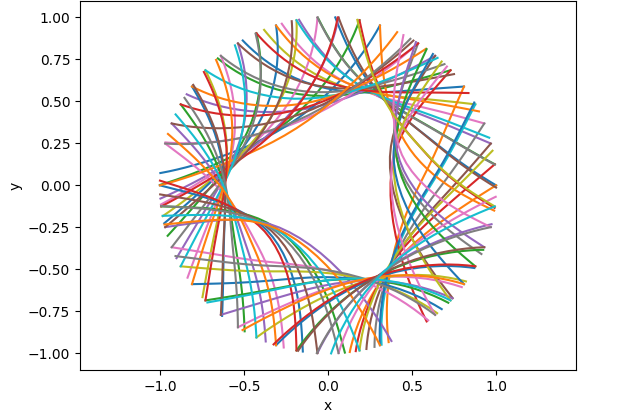}
	\caption{Left: The rays that hit the obstacle in (nearly) normal direction in Example 5. Right: The tangent rays detected from scattering relation.}
	\label{fig:ex5-rays}
\end{figure}
Then we use the stabilized adaptive phase space method in Section~\ref{sec:stab-adaptive} on all the non-broken rays to recover the metric as much as possible. The method converges to a solution with relative $L^2$ error of $2.90\times 10^{-2}$ at $9$th iteration. We plot the numerical and exact solutions in Figure~\ref{fig:ex5}.
\begin{figure}[!htb]
	\includegraphics[scale=0.35]{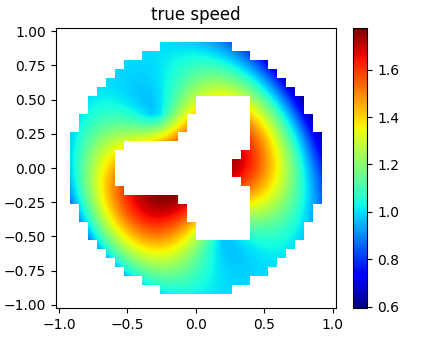}
	\includegraphics[scale=0.35]{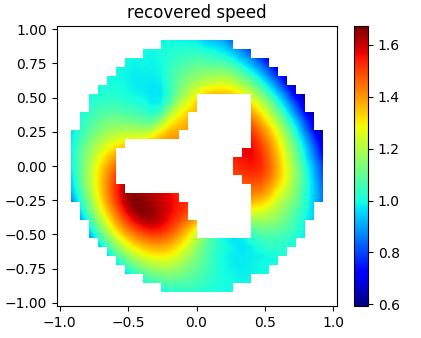}
	\includegraphics[scale=0.35]{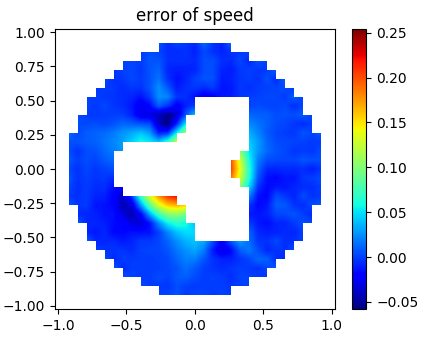}
	\caption{Left: exact solution. Middle: numerical solution at $9$th iteration. Right: error between the numerical solution and the exact solution.}
	\label{fig:ex5}
\end{figure}
After having recovered the metric from the non-broken rays' scattering relation, we can approximate the convex hull of the obstacle by the recovered non-broken rays, see Figure~\ref{fig:ex5-nonbroken-rec}. And by tracing the rays in collection $T$, we can approximately obtain the reflection points on the boundary of the obstacle, also see Figure~\ref{fig:ex5-nonbroken-rec}.  
\begin{figure}[!htb]
	\centering
	\includegraphics[scale=0.35]{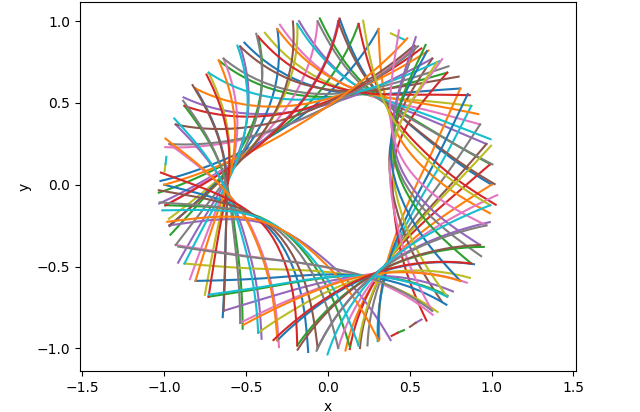}
	\includegraphics[scale=0.35]{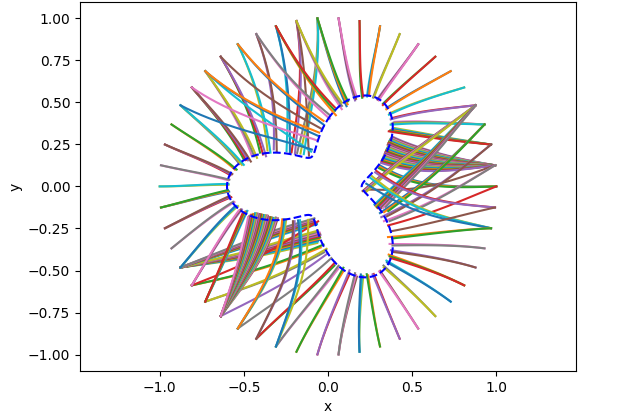}
	\caption{Left: The tangent rays computed from recovered metric in Example 5. Right: The traced rays in collection $T$ to half of traveltime. The blue dashed line is the exact boundary of obstacle.}
	\label{fig:ex5-nonbroken-rec}
\end{figure}